%% file: Arxiv_Tree.tex
\documentclass[11pt, oneside,reqno]{article}
\usepackage{geometry}                
\usepackage{amsmath,amsthm,mathrsfs, color}
\usepackage{amsfonts,amssymb}         
\usepackage{amstext, verbatim,epsfig}      
\usepackage{graphicx}
\usepackage{amssymb}
\usepackage{epstopdf, enumerate}
\usepackage{pgf,tikz}
\usepackage{float}
\usetikzlibrary{arrows}
\usepackage{hyperref}
\usepackage{cite}
\usepackage{epstopdf}
\usepackage{color}
\usepackage{transparent}
\usepackage{subfigure}

\usepackage{xcolor}
\usepackage[normalem]{ulem}

\hypersetup{colorlinks=true,pdfborder=000,  citecolor  = blue, citebordercolor= {magenta}}
\hypersetup{linkcolor=magenta}
\makeatletter
\let\reftagform@=\tagform@
\def\tagform@#1{\maketag@@@{(\ignorespaces\textcolor{purple}{#1}\unskip\@@italiccorr)}}
\renewcommand{\eqref}[1]{\textup{\reftagform@{\ref{#1}}}}
\makeatother

\makeatletter

\DeclareUrlCommand\ULurl@@{%
  \def\UrlLeft{\uline\bgroup}%
  \def\UrlRight{\egroup}}
\def\ULurl@#1{\hyper@linkurl{\ULurl@@{#1}}{#1}}
\DeclareRobustCommand*\ULurl{\hyper@normalise\ULurl@}
\makeatother

\def\lessim{\ \lower4pt\hbox{$
		\buildrel{\displaystyle <}\over\sim$}\ }
\def\gessim{\ \lower4pt\hbox{$\buildrel{\displaystyle >}
		\over\sim$}\ }

\newcommand{\indi}{\ensuremath{\boldsymbol 1}}

\newtheorem{lemma}{\bf Lemma}
\newtheorem{definition}{\bf Definition}
\newtheorem{theorem}{\bf Theorem}

\newtheorem{corollary}{\bf Corollary}
\newtheorem{remark}{\bf Remark}

\newtheorem{claim}{\bf Claim}

\newtheorem{proposition}{\bf Proposition}

\newenvironment{Proof of lemma}{\noindent{\bf Proof of Lemma}}{\hfill$\Box$\newline}
\newenvironment{Proof of theorem}{\noindent{\bf Proof of Theorem}}{\hfill{\footnotesize${\square}$}\newline}
\newenvironment{Proof of theorems}{\noindent{\bf Proof of Theorems}}{\hfill$\Box$\newline}
\newenvironment{Proof of proposition}{\noindent{\bf Proof of Proposition}}{\hfill$\Box$\newline}
\newenvironment{Proof of propositions}{\noindent{\bf Proof of Propositions}}{\hfill$\Box$\newline}
\newenvironment{Proof of exercise}{\noindent{\it Proof of Exercise:}}{\hfill$\Box$}

\begin{document}

	\title{Pemantle's min-plus binary tree}
	
	\author{
		Antonio Auffinger \thanks{Northwestern University. Email: tuca@northwestern.edu }
	    \and
		Dylan Cable \thanks{Stanford University. Email: dcable@stanford.edu}
	}
	\maketitle

\begin{abstract}
We consider a stochastic process that describes several particles interacting by either merging or annihilation. When two particles merge, they combine their masses; when annihilation occurs, only the particle of smallest mass survives. Particles start at the bottom of a binary tree of depth $N$ and move towards the root. Assuming that merging or annihilation happens independently at random, we determine the limit law of the final  mass of the system in the large $N$ limit.

\end{abstract}

\section{Introduction and main results}

Consider a binary tree $\mathcal T$ of depth $N$, $N\geq 1$, and a family of independent Bernoulli random variables $\{\eta_{v} \}_{v \in \mathcal T}$ indexed by the vertices of $\mathcal T$ with common distribution
\[
\mathbb P (\eta_{v}=1)= 1-\mathbb P (\eta_{v}=-1) = p,
\]  
where $p\in [0,1]$. If $\eta_{v}=1$, we place the operation $+$ (addition) at the vertex $v$ and, if $\eta_{v}=-1$, we place the operation $\min$ (minimum) at $v$. We use these operations recursively  starting from the leaves to construct a random variable $X_{N}$ at the root as follows. If $v$ is a leaf, we set $X_{v} = 1$. If $w$ is a vertex of the tree with children $v_{1}$, $v_{2}$, we set 
\begin{equation}\label{eq:recurrencemaster}
X_{w}= \begin{cases}
X_{v_{1}} + X_{v_{2}}, &\text{ if } \eta_{w}=1,\\
\min(X_{v_{1}},X_{v_{2}}), &\text{ if } \eta_{w}=-1. 
\end{cases}
\end{equation}
An example of the construction above is given in Figure \ref{fig:easy}.
Last, we set $X_{N}= X_{root}$. The goal of this paper is to answer the question: 
\begin{equation}\label{q:question}
\text{What is the typical syze of } X_{N} \text{ for } N \text{ large}?
\end{equation}

Question \eqref{q:question} (with $p=1/2$) was raised as an open problem by Robin Pemantle in his talk ``(Some of) my favorite (current) problems''  at the 2017 Southeastern Probability Conference. The value of $X_{N}$ models the total mass of a particle system or the capacity of a network flow through a sequence of merges or annihilations. Here, each vertex of $\mathcal T$ represents a channel of this network or the moment that two particles interact. Each channel either processes all incoming traffic (merging) or just takes the smallest input (selection). The stochastic process considered in this paper is a toy model that provides a rigorous framework for such flow/particle systems.

\begin{figure}
\scalebox{1.0}{ 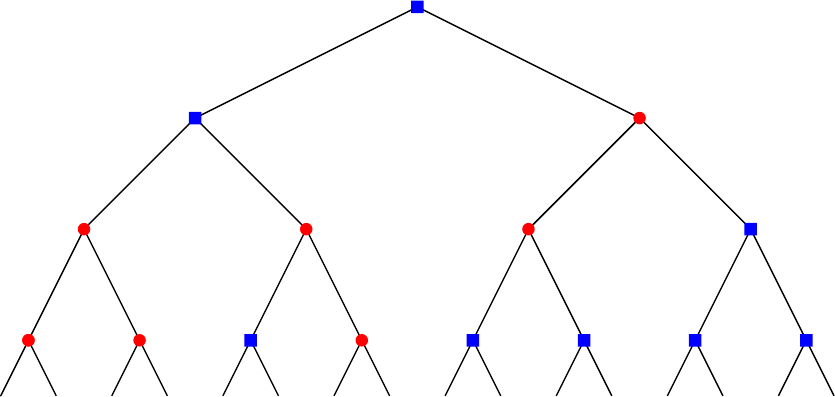}
 \centering
 \def \svgwidth{3000pt}
 \caption{The construction of the random variable $X_{N}$. Here $N=5$ and $X_{5}=4$. Red circles represent minimization while blue squares represent addition. }
 \label{fig:easy}
\end{figure}

The process defined through \eqref{eq:recurrencemaster} can be described by a random recursive relation. Random recurrence relations such as \eqref{eq:recurrencemaster} were also studied by Kesten \cite{Kesten}, Carmona-Petit-Yor \cite{CPY} and Collamore \cite{Colla} in the context of one dimensional or linear recursions. The equations considered in these papers (see also Aldous-Bandyopadhyay \cite{Aldous} and the references therein for related questions) were motivated by financial processes or by the study of random walks in random environments. In a random hierarchical lattice, this problem is related to the graph distance between two boundary points (see Hambly-Jordan \cite{Hambly}). The process $X_{N}$ is also related and inspired by the study of Boolean trees \cite{Pemantle, Gardy}. The reader is invited to check the paper of Pemantle-Ward \cite{Pemantle} and its references for more in this topic.

It is not difficult to determine the asymptotic behavior of $X_{N}$ when $p\neq 1/2$ (see Section \ref{sec:othercases}). If $p<1/2$ then $X_{N}$ is a tight family of random variables (Theorem \ref{thm:ASDpsdoihgui9fer}), while for $p>1/2$, $X_{N}$ grows exponentially with $N$. Our main theorem deals with the critical case of \eqref{q:question}, when $p=1/2$. 

\begin{theorem}\label{thm:distribution} Let $p=1/2$ and $c= \pi^{2}/3$. For all $t \in \mathbb R$,
\begin{equation*}
\lim_{N\to \infty} \mathbb P \left( \frac{1}{\sqrt{cN}}\log X_{N} \leq t\right) = 
\begin{cases}
0, \quad &\text{ if } \quad t \leq 0, \\
t^{2}, \quad &\text{ if } \quad 0\leq t \leq 1, \\
1, \quad &\text{ if } \quad t \geq 1.
\end{cases}
\end{equation*}
\end{theorem}

The theorem above implies that $X_{N}$ is of order $\exp(\sqrt{N})$; this scaling was conjectured by Robin Pemantle, although not the exact value of the constant $c$. From the above theorem, we also obtain the behavior of the expectation of $\log X_{N}$.

\begin{corollary}\label{thm:expectation} Let $p=1/2$. Then 
\begin{equation*}
\lim_{N\to \infty} \frac{1}{\sqrt{N}} \mathbb E \log X_{N}  = \frac{2\pi}{3\sqrt{3}}.
\end{equation*}
\end{corollary}

Our strategy for proving Theorem \ref{thm:distribution} is to study a recurrence relation for the distribution function of $X_{N}$. In order to describe this recurrence, let $N,k \geq 1$, and set 
\[
p_{N,k} := \mathbb{P} \left( X_N \geq k \right).
\]
Consider the two subtrees of depth $N$ attached to the root $r$ of a tree $\mathcal T_{N+1}$ of depth $N+1$. Let $v_{1}$ and $v_{2}$ be the two children coming from the root of the tree $\mathcal T_{N+1}$. Note that we can decompose the event $\{X_{N+1} \geq k\}$ into a union of disjoint events:
\begin{equation*}
\begin{split}
\{X_{N+1} \geq k\} = \{X_{v_{1}}\geq k, X_{v_{2}} \geq k\}& \cup \{\eta_{r} = +, X_{v_{1}} \geq k, X_{v_{2}} < k\}
\\ &\cup \{\eta_{r} = +, X_{v_{1}} < k, X_{v_{2}} \geq k\}
\\ &\cup  \{\eta_{r}=+,X_{v_{1}}=1,k>X_{v_{2}}\geq k -1\} 
\\ &\cup \{\eta_{r}=+,X_{v_{1}}=2,k>X_{v_{2}}\geq k -2\}
\\& \cdots \\
&\cup \{\eta_{r}=+,X_{v_{1}}=k-1,k>X_{v_{2}}\geq 1\}.
\end{split}
\end{equation*}
This decomposition yields the recurrence:
\begin{equation*} \label{eq:firstrecurrence}
\begin{split}
    p_{N+1,k} = p_{N,k}^2 + \frac{1}{2} \bigg(&2p_{N,k}(1-p_{N,k}) + (p_{N,1} - p_{N,2})(p_{N,k-1}-p_{N,k}) \\
    &\quad  + (p_{N,2} - p_{N,3})(p_{N,k-2}-p_{N,k})  + \cdots \\&\quad +(p_{N,k-1} - p_{N,k})(p_{N,1}-p_{N,k})\bigg),
\end{split}
\end{equation*}
that is,
\begin{equation} \label{eq:firstrecurrence2}
\begin{split}
    p_{N+1,k} - p_{N,k}= \frac{1}{2} \sum_{\ell =1}^{k-1} (p_{N,\ell}-p_{N,\ell+1})(p_{N,k-\ell}-p_{N,k}).
\end{split}
\end{equation}
The initial condition is given by
\[
p_{1,1} = 1, \quad p_{1,k}=0, \quad \text{ for } k>1. 
\]

Our strategy is to show that 
$$p_{N,k}=\mathbb{P}\bigg(X_N \geq k\bigg)  \sim 1 -\frac{\log(k)^2}{cn},$$
as $N$ gets large.
As we will see later on, the main observation is that the right hand side of the recurrence \eqref{eq:firstrecurrence2} has non-negative partial derivatives in the variables $p_{N,\ell}$, $1\leq \ell \leq k-1$. This allow us to change \eqref{eq:firstrecurrence2} to a sequence of inequalities and this approach provides upper and lower bounds for $p_{N+1,k}$.

We end this section commenting on previous unpublished work on this question. These remarks were communicated to us by Robin Pemantle \cite{Pemantle2} and are not used in this paper. First, Albert Chen proved that $\mathbb P(X_{N}=k)$ decreases with $k$ for any $N$ fixed. This fact was independently discovered by Tam Cheetham-West at PCMI. Second, Thomas Duquesne established the upper bound 
\[
\lim_{N\to \infty} \mathbb P \bigg( \log X_{N}/\sqrt{N} \leq 2 \bigg) = 1.
\]
Last, a lower bound was recently discovered by Jian Ding: 
\[ \mathbb P\bigg(\log X_{N}/\sqrt{N} \leq \varepsilon\bigg)< g(\varepsilon), 
\]
for some function $g$ so that $g(\varepsilon) \to 0$ as $\varepsilon \to 0$. 

\subsection{Acknowledgments}
Both authors would like to thank Robin Pemantle for sharing his initial notes and explaining the initial motivation for this problem. The authors also thank Nicolas Curien for pointing out reference \cite{Hambly}. This project started at the 2017 IAS PCMI summer school in Park City, Utah. The school was supported by NSF grant DMS 1441467. The research of A.A. is supported by NSF CAREER grant DMS 1653552. The research of D.C. is supported by a UAR major grant from Stanford University.  D.C. would like to thank the department of mathematics at Northwestern University for hospitality during his visit.  D.C. would like to thank Mark Selke, Tam Cheetham-West, Kai Hugtenburg, Emilia Alvarez, Terrence Tsui and Polina Belopashentseva for initial discussions on this problem at PCMI. 

\section{Upper Bound for $p_{N,k}$}

In this section, we establish an upper bound for $p_{N,k}$. Our main goal is to prove the following proposition.
\begin{proposition}\label{prop:upperbound} 
Let $C>\pi^{2}/3$. There exists $\beta>1$  and $N_{0}$ such that for all $N \geq 1$, if we let $n = N + N_0$ and $k \geq 1$,
\begin{equation}\label{cownunchuckpudding}
p_{N,k} \leq  
 \begin{cases} 
   1-\frac{\log(k)^2}{nC} &\text{ if } \log k < \sqrt{nC+\beta^2} - \beta, \\
      (\frac{2\beta\sqrt{nC+\beta^2}-2\beta^2}{Cn})e^{-\beta^{-1}(\log k - (\sqrt{nC+\beta^2} - \beta))} &\text{ if } \log k \geq \sqrt{nC+\beta^2} - \beta.
   \end{cases}
\end{equation}  
\end{proposition}

We first explain the main construction behind the proof of the proposition above. Suppose we choose an array of numbers $\{q_{N,k}\}_{N, k\geq 1}$, that  satisfies the following inequality for $N \geq N_0, k \geq 1$:
\begin{equation}\label{rec:UB}
q_{N+1,k} - q_{N,k} \geq \frac{1}{2} \sum_{\ell =1}^{k-1} (q_{N,\ell}-q_{N,\ell+1})(q_{N,k-\ell}-q_{N,k}),
\end{equation}
and such that 
\begin{equation}\label{rec:UBic}
q_{N_{0},k} \geq p_{1,k}= \indi \{ k=1 \} \text{ for all } k\geq 1.
\end{equation}

Furthermore, assume that for each $N$ and $k$, the vector $(q_{N,1}, \ldots, q_{N,k})$ belongs to the set 
\[
\mathcal S_{k} = \bigg\{(x_{1}, \ldots, x_{k})\in \mathbb{R}^k \bigg |1 \geq x_1 \geq x_2 \geq \dots \geq x_k \geq 0\bigg\}.
\] 
 Then, we claim that by induction
 \begin{equation}\label{eq:induction}
 q_{N_{0}+N,k} \geq p_{1+N,k}, \text{ for all } k \geq 1, N \geq 1.
 \end{equation}
Indeed, define the function 
\begin{equation}\label{def:mainfunction}
    f(x_1,x_2,\dots,x_k) := x_k + \frac{1}{2} \sum_{\ell =1}^{k-1} (x_{\ell}-x_{\ell+1})(x_{k-\ell}-x_{k}).
\end{equation}
The function $f$ has partial derivatives given by
\begin{equation}\label{eq:PD}
    \frac{\partial f}{\partial x_j} = \begin{cases}
    1 - x_1 + x_k, & \text{ if } j = k, \\ 
    x_{k-j} - x_{k-j+1}, &\text{ if } j < k,
    \end{cases}
\end{equation}
which are non-negative in the domain $\mathcal S_{k}$.
 
 Now assume that for some $N$, $q_{N_0+N,j} \geq p_{1+N,j}$ for all $j \leq k$.  Using the recurrence \eqref{rec:UB},  we obtain
\begin{equation*}
\begin{split}
    q_{N_{0}+N+1,k} &\geq q_{N_{0}+N,k} + \frac{1}{2} \sum_{\ell =1}^{k-1} (q_{N_{0}+N,\ell}-q_{N_{0}+N,\ell+1})(q_{N_{0}+N,k-\ell}-q_{N_{0}+N,k})\\
&=f(q_{N_{0}+N,1}, \ldots, q_{N_{0}+N,k}) \\
&\geq f(p_{1+N,1}, \ldots, p_{1+N,k}) \\
&=p_{N+2,k},
\end{split}
\end{equation*}
where in the last inequality we used the inductive hypothesis and the fact that the partial derivatives of $f$ are non-negative in $\mathcal S_{k}$. The last equality is \eqref{eq:firstrecurrence2}. Hence, this calculation combined with \eqref{rec:UBic} proves our claimed inequality \eqref{eq:induction}. To sum up, we have proven:

\begin{lemma}\label{lem:lemUB}
Assume that $\{q_{N,k}\}_{k,N\geq 1}$ is an array that satisfies \eqref{rec:UB} and \eqref{rec:UBic} and such that $(q_{N,1}, \ldots, q_{N,k}) \in \mathcal S_{k}$ for all $N$ and $k$. Then \eqref{eq:induction} holds. In this case, we say that the array $q_{N,k}$ is an upper bound for $p_{N,k}$. 
\end{lemma}

\begin{remark}\label{rem:rem11}
If the initial condition \eqref{rec:UBic} is replaced by 
\begin{equation*}
q_{N_{0},k} \geq p_{1,k} \quad \text{ for } \quad 1 \leq k \leq K_{0},
\end{equation*}
and inequality \eqref{rec:UB} is only required to hold for $k \leq K_0$,
then the proof of Lemma \ref{lem:lemUB} implies 
\[
q_{N_{0}+N,k} \geq p_{1+N,k}, \text{ for all } N \geq 1 \text{ and } 1 \leq k \leq K_{0}.
\]
In this case, we say that $q_{N,k}$ is an upper bound for $p_{N,k}$ when $k\leq K_{0}$.
\end{remark}

Let us now use this strategy to show an upper bound for small $k \leq K_0$.  

\begin{lemma}\label{lem:Lemma2} For any $c > \pi^{2}/3$ there exists a $N_0 \geq 1$ such that for $N\geq N_{0}$
\[ q_{N,k} := 1 - \frac{\log(k)^2}{cN} \] 
is an upper bound for $p_{N,k}$ when $k \leq K_0 = 150$.
\end{lemma}

\begin{proof}
Note that for $k = 1$, we check directly that $q_{N,1} = 1 $ is an upper bound for $p_{N,1}$. Now, we define a sequence of constants $b_k, 1\leq k \leq 150$ as follows. Let $b_1 = 0$ and for $k > 1$, let $b_k$ be the positive root of the polynomial: 
\begin{equation}\label{eq:polynomialbk}
    b_k^2-2b_k-\big((b_2-b_1)b_{k-1}+(b_3-b_2)b_{k-2}+\cdots+(b_{k-1}-b_{k-2})b_2\big) = 0.
\end{equation}
Letting $d_k$ denote the constant term in \eqref{eq:polynomialbk}, that is,
\begin{equation*}
    d_k = (b_2-b_1)b_{k-1}+(b_3-b_2)b_{k-2}+\cdots+(b_{k-1}-b_{k-2})b_2,
\end{equation*}
we obtain
\begin{equation*}
    b_k = 1 + \sqrt{1+d_k}.
\end{equation*}
The sequence $b_{k}$ for $1\leq k \leq 150$ can be explicitly computed with $b_2 = 2$, and $b_3 = 1 + \sqrt{5}$. It can be verified numerically that for this range of values of $k$, 
\[
b_k > \frac{3\log(k)^2}{\pi^{2}}.
\] 
The statement of Lemma \ref{lem:Lemma2} will now follow once we prove that for any choice of $\{ a_{k}\}_{1 \leq k \leq 150}$ with $0 < a_k < b_k$, there exists $N_0 = N_{0}(\{ a_{k}\})$ such that  
\[ p_{N,k} \leq 1 - \frac{a_k}{N+N_0} \quad \text { for all } N\geq 1.
\] 

We prove this fact by induction on $k$. Let $k > 1$, and assume the result holds for $1,\dots,k-1$. Then, without loss of generality we can assume that there exists a $\delta > 0$, such that  $a_{\ell}=b_{\ell}-\delta$, for $1\leq \ell \leq k-1$.

Set
\begin{equation*}
    q'_{N,j} = 1 - \frac{a_j}{N} = 1 - \frac{b_j-\delta}{N}, 1 \leq j \leq k-1 \quad \text { and } \quad  q'_{N,k}=1-\frac{a_{k}}{N}.
\end{equation*}
If we can show that for this $k$, inequality \eqref{rec:UB} holds for the array $\{q_{N,j}'\}_{1 \leq j \leq k}$  and for large $N \geq N_1$, then it will follow by Remark \ref{rem:rem11} that 
\[ p_{N,k} \leq q_{N+N_{2},k} \quad \text { for all } N\geq 1, 
\] 
with $N_{2}=\max\{N_{0},N_{1}\}$, as desired.

Calculating the right side of inequality \eqref{rec:UB} for $q'_{N,k}$, we get
\[
\begin{split}
    \mathcal Z :&= \frac{1}{2N^2}\bigg((a_2-a_1)(a_{k-1}-a_k) + (a_3-a_2)(a_k-a_{k-2})+\cdots+(a_k-a_{k-1})(a_k-a_1)\bigg) \\
    &=\frac{1}{2N^2}\bigg(2a_k +(a_k^2-2a_k-((b_2-b_1)(b_{k-1}-\delta)+(b_3-b_2)(b_{k-2}-\delta)\\
    &\quad \quad \quad \quad \quad \quad \quad \quad \quad \quad \quad \quad +\cdots+(b_{k-1}-b_{k-2})(b_{2}-\delta)-(b_{k-1}-\delta)(b_1-\delta))\bigg) \\
&=\frac{1}{2N^{2}}(2a_{k} + \Psi(\delta)).
\end{split}
\]

Define
\begin{equation*}
    g_{k}(x) = x^2 - 2x - d_{k}.
\end{equation*}
Now, we claim that if we choose $\delta$ small enough, we have $\Psi(\delta)<0$.
Indeed, $\Psi$ is continuous, and $\Psi(0) = g_{k}(a_k) < 0$
 since $g_{k}$ is a quadratic polynomial with positive first term, and $a_k$ lies between the two roots of $g_{k}$ by assumption.  Choosing $\delta$ such that $\Psi(\delta)<0$, and setting $\Psi(\delta) = -2\varepsilon < 0$ for some $\varepsilon>0$ then
\begin{equation*}
    \mathcal Z = \frac{1}{2N^2}(2a_k - 2\varepsilon) = \frac{1}{N^2}(a_k - \varepsilon) \leq \frac{a_k}{N(N+1)} =  q'_{N+1,k}-q'_{N,k}.
\end{equation*}
Here, the inequality holds for large $N \geq N_1$. Thus, inequality \eqref{rec:UB} holds for $q_{N,k}'$.  If necessary we make $N_1$ larger so that $q'_{N_1,k} \geq 0$. This ends the proof of Lemma \ref{lem:Lemma2}. 
\end{proof}

\subsection{Proof of Proposition \ref{prop:upperbound}}

Fix $\delta>0$ and set $C= (1+\delta) \pi^{2}/3$. We will soon choose $\beta=\beta(\delta)>0$ and define  
\begin{equation}\label{def:qnk}
q_{N,k} =
\begin{cases} 
   1-\frac{\log(k)^2}{NC} &\text{ if } \log k < \sqrt{NC+\beta^2} - \beta, \\
      (\frac{2\beta\sqrt{NC+\beta^2}-2\beta^2}{NC})e^{-\beta^{-1}(\log k - (\sqrt{NC+\beta^2} - \beta))} &\text{ if } \log k \geq \sqrt{NC+\beta^2} - \beta.
\end{cases}
\end{equation}
Our goal in this section is to show that for this choice of $q_{N,k}$ we have some $N_{0} \geq 1$ such that
\begin{equation}\label{eq:upperbounda?}
p_{N,k} \leq q_{N_{0}+N,k}, \quad \text{ for all } N\geq 1, k\geq 1.
\end{equation}

Note that by Lemma \ref{lem:Lemma2}, we know that \eqref{eq:upperbounda?} holds for $k\leq 150$ and all $N\geq 1$. In what follows, we will always take $N\geq N_{0}$ and increase the choice of $N_{0}$ when necessary without mentioning. This will happen a finite number of times. Last, we will say that the pair $(N,k)$ (or the variable $q_{N,k}$) is {\it from model $1$} if  \[\log k < \sqrt{NC+\beta^2} - \beta.\]Otherwise, we say that $(N,k)$ is {\it from model} $2$. 

We will break the proof of \eqref{eq:upperbounda?} into three cases: $(N,k)$ and $(N+1,k)$ are from model 1, $(N,k)$ and $(N+1,k)$ are both from model 2, or $(N,k)$ is from model $2$ and $(N+1,k)$ is from~$1$. 
\begin{lemma}\label{lem:lemadad}
If $\log(k) < \sqrt{NC+\beta^2} - \beta$, the recurrence \eqref{rec:UB} is satisfied for $q_{N,k}$ as in~\eqref{def:qnk}. Furthermore, we have 
\begin{equation}\label{eq:extralemma3}
q_{N+1,k} - q_{N,k} - \frac{1}{2} \sum_{\ell =1}^{k-1} (q_{N,\ell}-q_{N,\ell+1})(q_{N,k-\ell}-q_{N,k}) \geq \frac{\gamma(C)(\log k)^{2}}{N^{2}},
\end{equation}
where $ \gamma(C) >0$ for all $C>\pi^{2}/3.$

\end{lemma}
In the proof of the above lemma we will need the following fact.
\begin{lemma}\label{lem:calculuseries}
Let \[h(k)= \sum_{\ell=1}^{k-1} \frac{1}{\ell}\log\left(\frac{k}{k-\ell}\right).\]
Then $h(k) \leq \frac{\pi^2}{6}$ for all $k \geq 1$ and 
\[\lim_{k\to \infty} h(k) = \frac{\pi^{2}}{6}.\]
\end{lemma}
\begin{proof}
First, by Taylor's theorem,
\begin{equation*}
    h(k) = \sum_{\ell=1}^{k-1} \frac{1}{\ell}\bigg(\frac{\ell}{k}+\frac{\ell^2}{2k^2}+\frac{\ell^3}{3k^3} + \cdots\bigg).
\end{equation*}
Switching the order of summation (since all terms are positive), we obtain 
\begin{equation}\label{eq:lemmacalc}
    h(k) = \sum_{a = 1}^{\infty}\frac{1}{ak^a}\sum_{\ell=1}^{k-1}\ell^{a-1}.
\end{equation}
As
\begin{equation*}
    \sum_{\ell=1}^{k-1}\ell^{a-1} \leq \int_1^kx^{a-1}dx \leq \frac{k^a}{a},
\end{equation*}
we have
\begin{equation*}
    h(k) \leq \sum_{a = 1}^{\infty}\frac{1}{a^2} = \frac{\pi^2}{6},
\end{equation*}
which shows the first fact.  The second fact follows directly by the dominated convergence theorem. Indeed, by Faulhaber's generalized formula for sum of powers \cite[Page 106]{Conway}, we have
\begin{equation*}
     \sum_{\ell=1}^{k-1}\ell^{a-1} = \frac{k^a}{a} + r(k),
\end{equation*}
where $r(k)$ is a polynomial of degree $(a-1)$. Therefore, the $a^{th}$ term in \eqref{eq:lemmacalc} is given by
\begin{equation*}
    \sum_{\ell=1}^{k-1}\frac{\ell^{a-1}}{ak^a} = \frac{1}{a^2} + \frac{r(k)}{ak^a} \to_{k \to \infty} \frac{1}{a^2}.
\end{equation*}
\end{proof}
\begin{proof}[Proof of Lemma \ref{lem:lemadad}]

In this case, all $q_{N,j},$ for $ j \leq k$ and $q_{N+1,k}$ are from model $1$. Thus 
\begin{equation}\label{lemmagoal}
q_{N+1,k}-q_{N,k} = \frac{(\log k)^2}{CN(N+1)} \geq \frac{3(\log k)^2}{(1+\delta)\pi^{2}N^2} 
\end{equation}
for large $N$.
Set
\begin{equation*}
    X = \frac{1}{2}\sum_{\ell=1}^{k-1} (q_{N,\ell}-q_{N,\ell +1})(q_{N,k-\ell}-q_{N,k}),
\end{equation*}
so that
\begin{equation}\label{eq:Xlemadad}
\begin{split}
    S:=2C^2N^2X =  \sum_{\ell=1}^{k-1} (\log(\ell+1)^2 - \log(\ell)^2)(\log(k)^2 - \log(k-\ell)^2). 
\end{split}
\end{equation}
Set  \[e_\ell = \log(\ell+1)^2-\log(\ell)^2 - 2\frac{\log(\ell)}{\ell}. \]
Since $\frac{\partial^{2}}{\partial x^{2}}(\log x)^{2} <0$ for all $x\geq3$, $e_{\ell}<0$ for $\ell\geq 3$. Thus, as $\max \{ (\log 2)^{2}, (\log 3)^{2}-(\log 2) \} \leq 1$,

\begin{equation*}
\begin{split}
    S &\leq \sum_{\ell=3}^{k-1} \bigg(\frac{2\log\ell}{\ell}\bigg)((\log k)^2 - \log(k-\ell)^2) + 2(\log k)^2 - \log(k-1)^2  - \log(k-2)^2 \\
    &+ \sum_{\ell=3}^{120}e_\ell((\log k)^2 - \log(k-\ell)^2)\\
    & \leq \sum_{\ell=3}^{k-1} \bigg(\frac{2\log(\ell)}{\ell}\bigg)((\log k)^2 - \log(k-\ell)^2) +  \frac{6\log(k-2)}{k-2}+ \sum_{\ell=3}^{120} e_{\ell} \frac{2 \ell \log k}{k}\\
    &\leq  \sum_{\ell=3}^{k-1} \bigg(\frac{2\log(\ell)}{\ell}\bigg)((\log k)^2 - \log(k-\ell)^2)+ \bigg(7 + \sum_{\ell=3}^{120}2\ell e_\ell\bigg)\frac{\log(k)}{k}. 
    \end{split}
\end{equation*}
A direct computation shows that $\sum_{\ell=3}^{120}2\ell e_\ell < -7$, so the sum above is bounded by 
\[
 \sum_{\ell=3}^{k-1} \bigg(\frac{2\log(\ell)}{\ell}\bigg)((\log k)^2 - \log(k-\ell)^2).
\] 
Using that 
\[ 
 (\log k)^2-\log(k-\ell)^2 = ( \log k)^2 - \bigg(\log k + \log\left(\frac{k-\ell}{k}\right)\bigg)^{2} \leq -2\log(k)\log\left(\frac{k-\ell}{k}\right),
\]
we conclude that for $k \geq 150$,
\[ 
\begin{split}
  S &\leq -\sum_{\ell=3}^{k-1} \frac{2\log \ell }{\ell} \bigg(2\log k \log\left(\frac{k-\ell}{k}\right)\bigg) \\
  &\leq 4(\log k)^2\sum_{\ell=3}^{k-1} \frac{1}{\ell}\log\left(\frac{k}{k-\ell}\right) \leq 4 (\log k)^{2} \frac{\pi^{2}}{6}, \\ 
  \end{split}
\]
where in the last line we used Lemma \ref{lem:calculuseries}. As a result, we obtain from \eqref{lemmagoal} and \eqref{eq:Xlemadad} that 
\begin{equation}\label{eq:differenceLem3}
X \leq  (\log k)^{2} \frac{\pi^{2}}{3C^{2}N^{2}} \leq  \frac{(\log k)^{2}}{(1+\delta)CN^{2}} \leq \frac{(\log k)^2}{CN(N+1)} =q_{N+1,k} - q_{N,k}, 
\end{equation}
so inequality \eqref{rec:UB} holds for this case.
To check \eqref{eq:extralemma3} just take the difference in \eqref{eq:differenceLem3}.
\end{proof}

The next lemma deals with the intermediate case.
\begin{lemma}\label{lem:middle}
Inequality \eqref{rec:UB} holds when 
\begin{equation*}\label{eq:midcase}
 \sqrt{NC+\beta^2} - \beta \leq  \log k < \sqrt{(N+1)C+\beta^2} - \beta.  
\end{equation*}
\end{lemma}
\begin{proof}
In this case we have that $(N+1,k)$ is from model 1, but $(N,k)$ is from model 2.  Our strategy is to recall that inequality $\eqref{rec:UB}$ holds when all variables are sourced from model 1, as in Lemma \ref{lem:lemadad} and then change  the variables $q_{N,k-\ell}$ to model 2, for $\ell$ such that
\[
\log(k-\ell) \geq \sqrt{NC+\beta^2} - \beta.
\]
Let $L$ be the largest of such $\ell$'s. Note that $L\leq k/3.$
Define
\begin{equation*}
    q_{N,\ell}^{(1)} = 1 - \frac{(\log \ell)^2}{NC},
\end{equation*}
\begin{equation*}
    q_{N,\ell}^{(2)} = (\frac{2\beta\sqrt{NC+\beta^2}-2\beta^2}{NC})e^{-\beta^{-1}(\log \ell - (\sqrt{NC+\beta^2} - \beta))}, \quad \text{ for } 0\leq k-\ell\leq L,
\end{equation*}
$q_{N,\ell}^{(2)}=q_{N,\ell}^{(1)}$ for $k-\ell>L$, and set \[
Y^{(i)} = f(q_{N,0}^{(i)},\ldots, q_{N,k}^{(i)}),\]
for $i=1,2$ and $f$ given in \eqref{def:mainfunction}.

By \eqref{eq:extralemma3}
\begin{equation*}
    q_{N+1,k} - Y^{(1)} > \gamma(C)\frac{\log(k)^2}{N^2} \geq \frac{\gamma(C)}{2N},
\end{equation*}
for $N$ large enough.
As a result, it suffices to show
\begin{equation*}
    Y^{(2)} - Y^{(1)} \leq \frac{\gamma(C)}{2N}.
\end{equation*}
We bound $Y^{(2)} - Y^{(1)}$ by considering the changes due to each variable $q_{N,k-\ell}$ for $0 \leq \ell \leq L$,
\begin{equation*}
    Y^{(2)} - Y^{(1)} \leq \sum_{\ell=0}^{L}(q_{N,k-\ell}^{(2)} - q_{N,k-\ell}^{(1)})\frac{\partial f}{\partial q_{N,k-\ell}}(\xi_{\ell}), 
\end{equation*}
for some $q_{N,k-\ell}^{(1)} \leq \xi_{\ell} \leq q_{N,k-\ell}^{(2)}$.
Using \eqref{eq:PD}, we have
\[
\begin{split}
    Y^{(2)} - Y^{(1)} &\leq \sum_{\ell=0}^{L}(q_{N,k-\ell}^{(2)} - q_{N,k-\ell}^{(1)})\frac{\partial f}{\partial q_{N,k-\ell}}(\xi_{\ell}) \leq \max_{\ell=0}^L (q_{N,k-\ell}^{(2)} - q_{N,k-\ell}^{(1)})\sum_{\ell=0}^{L}\frac{\partial f}{\partial q_{N,k-\ell}}(\xi_{\ell})
     \\  &\leq\max_{\ell=0}^L (q_{N,k-\ell}^{(2)} - q_{N,k-\ell}^{(1)})(1+\sum_{\ell=1}^{L} \xi_{k-\ell} - \xi_{k-\ell-1}) \leq 2\max_{\ell=0}^L (q_{N,k-\ell}^{(2)} - q_{N,k-\ell}^{(1)}).
\end{split}
\]
We now end the proof of Lemma \ref{lem:middle} by claiming that for $N$ large enough
\[
\max_{\ell=0}^L (q_{N,k-\ell}^{(2)} - q_{N,k-\ell}^{(1)}) \leq \frac{\gamma(C)}{4N}.
\]
Indeed, by making the change of variables $m=\log (k-\ell)$, we see that computing $\max_{\ell=0}^L (q_{N,k-\ell}^{(2)} - q_{N,k-\ell}^{(1)})$ is equivalent to minimizing 
\begin{equation*}
g(m) = 1 - \frac{m^2}{CN} - \frac{2\beta\sqrt{NC+\beta^2}-2\beta^2}{NC}e^{-\alpha_1(m - (\sqrt{NC+\beta^2} - \beta))},
\end{equation*}
 over the range $\sqrt{NC+\beta^2} -\beta \leq m < \sqrt{(N+1)C + \beta^2} - \beta$. Thus, it suffices to show 
 \[ 
\min_{\sqrt{NC+\beta^2} -\beta \leq m < \sqrt{(N+1)C + \beta^2}-\beta}  g(m) \geq -  \frac{\gamma(C)}{4N}.
 \]
 
Notice that $g$ is concave with 
$$g''(m) = \frac{-2}{CN} - \frac{2\beta\alpha_1^2\sqrt{NC+\beta^2}-2\beta^2}{NC}e^{-\alpha_1(m - (\sqrt{NC} - \epsilon))} < 0.$$
As a result, it's minimum over the range $\sqrt{NC+\beta^2} - \beta \leq m \leq \sqrt{(N+1)C+\beta^2} - \beta$ must occur at one of the two endpoints.
That is, for all $m$, 
\[
g(m) \geq \min\{g(\sqrt{NC+\beta^2} - \beta),g(\sqrt{(N+1)C+\beta^2} - \beta)\}.
\] 
Now, we calculate the value of $g$ at these two endpoints. At the first endpoint, the two models are by design chosen to be equal, that is,
\begin{equation*}
g(\sqrt{NC+\beta^2} - \beta) = 1 - \frac{NC - 2\beta\sqrt{NC+\beta^2} + 2\beta^2}{NC} - \frac{2\beta\sqrt{NC+\beta^2}-2\beta^2}{NC} = 0.
\end{equation*}
At the second endpoint, we have
\begin{equation*}
\begin{split}
    g(&\sqrt{(N+1)C + \beta^2} - \beta)\\ 
    &\geq1 - \frac{(N+1)C - 2\beta\sqrt{(N+1)C+\beta^2} + 2\beta^2}{NC}- \frac{2\beta\sqrt{NC+\beta^2}-2\beta^2}{NC}e^{-\frac{C\alpha}{2\sqrt{(N+1)C+\beta^{2}}}} \\
    &\geq -\frac{1}{N} + \frac{-2\beta^2}{NC} + \frac{2\beta\sqrt{(N+1)C+\beta^2}}{NC}
    \\ &\quad - \frac{2\beta\sqrt{NC+\beta^2}-2\beta^2}{NC}\bigg(1 - \frac{\alpha_1C}{2\sqrt{(N+1)C+\beta^2}} + O\bigg(\frac{1}{N}\bigg)\bigg)  
    \\ &\geq \frac{-1}{N}+\frac{-2\beta^2}{NC} + \frac{2\beta^2}{NC} + \frac{2\beta\sqrt{(N+1)C+\beta^2}}{NC} - \frac{2\beta\sqrt{NC+\beta^2}}{NC} 
    \\\ &\quad +  
    \frac{\sqrt{NC+\beta^2}}{n\sqrt{(N+1)C+\beta^2}} + O\bigg(\frac{1}{N^{3/2}}\bigg)  \\ 
    &\geq \frac{-1}{N}(|2\beta^2/C-1|\frac{C}{\sqrt{NC+\beta^2}}) +O\left(\frac{1}{N^{3/2}}\right) = O\left(\frac{1}{N^{3/2}}\right) \geq  \frac{-\gamma(C)}{4N},
\end{split}
\end{equation*}
for $N$ large enough.  Thus, we have the desired bound for $g(m)$ and this concludes that recurrence \eqref{rec:UB} is satisfied in this case.
\end{proof}

We now deal with the case when both $(N,k)$ and $(N+1,k)$ are from model $2$, that is, $\log k \geq \sqrt{(N+1)C+\beta^2} - \beta$. Let $\alpha = \beta^{-1}$.
In this case, we can write 
\begin{equation}\label{def:theta}
q_{N,k} = \frac{2\beta(\sqrt{NC+\beta^2}-2\beta^2)}{CN}e^{\alpha(\sqrt{NC+\beta^2}-\beta)}\frac{1}{k^\alpha} = \theta_{N}\frac{1}{k^{\alpha}},
\end{equation}
for some constant $\theta_{N}$ that does not depend on $k$.
\begin{lemma}\label{lem:tailmodelUB}
Inequality \eqref{rec:UB} holds when 
\begin{equation*}
 \sqrt{(N+1)C+\beta^2} - \beta \leq \log k.   
\end{equation*}
\end{lemma}
\begin{proof}
Note that in this case, $q_{N,k}$ is also from model $2$.
As before, we need to show $q_{N+1,k} \geq q_{N,k} + X$ with $X$ as in \eqref{eq:Xlemadad}. Set 
\[ 
S(k) := \sum_{\ell=1}^{k-1}\left(\frac{1}{\ell^\alpha} - \frac{1}{(\ell+1)^\alpha}\right)\left(\frac{1}{(k-\ell)^\alpha} - \frac{1}{k^\alpha}\right). 
\]

If all $q_{N,\ell}$'s were from model $2$, then we would have $ X= \theta_{N}^{2}S(k)/2$. Our strategy here is to first get a bound for $S(k)$ as above and then see the effects of changing $q_{N,\ell}$'s to their true values.
\begin{lemma}\label{lem:SRSS}
For any $\varepsilon >0$, there exists $\alpha_{0}>0$ such that for all $0<\alpha<\alpha_{0}$, $k\geq 1$,
\[
S(k) \leq \frac{\alpha^2}{k^{2\alpha}}(1+\varepsilon)\frac{\pi^2}{6}.
\]

\end{lemma}
\begin{proof}[Proof of Lemma \ref{lem:SRSS}]
Let $0 < \alpha < 1/2$. Note that 
\begin{equation*}
    S(k) \leq \sum_{\ell=1}^{k-1}\frac{\alpha}{\ell^{1+\alpha}}\bigg(\frac{1}{(k-\ell)^\alpha} - \frac{1}{k^\alpha}\bigg) = \sum_{\ell=1}^{k-1}\frac{\alpha}{\ell^{1+\alpha}}\frac{1}{k^{\alpha}}\bigg((1-\frac{\ell}{k})^{-\alpha}-1\bigg).
\end{equation*}
Now, we expand via binomial series as $\ell < k$ to obtain
\begin{equation*}
    S(k) \leq \sum_{\ell=1}^{k-1}\frac{\alpha}{\ell^{1+\alpha}}\frac{1}{k^{\alpha}}\bigg(\frac{\alpha \ell}{k}+\frac{(-\alpha)(-\alpha-1)\ell^2}{2k^2} + \cdots\bigg).
\end{equation*}
All terms in the above series are positive, so we can switch the order of summation. Setting $\gamma_j = (-1)^{j-1}(-\alpha-1)\cdots(-\alpha-(j-1))$, and bounding the series, we have 
\begin{equation*}
\begin{split}
    S(k) &\leq \sum_{j=1}^{\infty}\sum_{\ell=1}^{k-1} \frac{\alpha^2\gamma_j}{j!\ell^{1-j+\alpha}k^{j+\alpha}} \\
    &\leq \sum_{j=1}^{\infty}\frac{\alpha^2\gamma_j}{j!k^{j+\alpha}}\sum_{\ell=1}^{k-1}\ell^{j-1-\alpha} \leq \sum_{j=1}^{\infty}\frac{\alpha^2\gamma_j}{j!k^{j+\alpha}}\frac{1}{j-\alpha}k^{j-\alpha} = \sum_{j=1}^{\infty}\frac{\alpha^2\gamma_j}{j!k^{2\alpha}}\frac{1}{j-\alpha}.
\end{split}
\end{equation*}

Now,
\begin{equation*}
\begin{split}
    S(k) &\leq  \frac{\alpha^2}{k^{2\alpha}}\sum_{j=1}^{\infty}\frac{1}{j^2}\frac{j+2\alpha}{j}\prod_{i=1}^{j-1}\frac{i+\alpha}{i} \\
    &\leq \frac{\alpha^2}{k^{2\alpha}}\sum_{j=1}^{\infty}\frac{1}{j^2}e^{2\alpha}\prod_{i=1}^{j-1}e^{\alpha/i} = \frac{\alpha^2}{k^{2\alpha}}\sum_{j=1}^{\infty}\frac{1}{j^2}e^{2\alpha}e^{\alpha H_{j-1}}\\
    &\leq \frac{\alpha^2}{k^{2\alpha}}\sum_{j=1}^{\infty}\frac{1}{j^2}e^{2\alpha}e^{\alpha(1+\log(j))} = \frac{\alpha^2}{k^{2\alpha}}\sum_{j=1}^{\infty}\frac{1}{j^{2-\alpha}}e^{3\alpha}.
  \end{split}
\end{equation*}
Here, $H_n$ is the $n^{th}$ harmonic number.  We notice that the series 
\[
\sum_{j=1}^{\infty}\frac{1}{j^{2-\alpha}}e^{3\alpha}
\]
is summable for $\alpha < 1$, so we can apply the monotone convergence theorem to conclude that it converges to $\frac{\pi^2}{6}$ as $\alpha \to 0$. This means that for any $\varepsilon >0$ we can choose $\alpha$ small enough so that
\begin{equation*}
    S(k) \leq \frac{\alpha^2}{k^{2\alpha}}(1+\varepsilon)\frac{\pi^2}{6}.
\end{equation*}
\end{proof}
Now, note that by Lemma \ref{lem:SRSS}, and \eqref{def:theta}

\begin{equation*}
  \frac{\theta_N^2S(k)}{2} \leq \frac{\theta_N^2\alpha^2\pi^2(1+\varepsilon)}{12k^{2\alpha}} = \frac{\alpha^2q_{N,k}^2\pi^2(1+\varepsilon)}{12}.
\end{equation*}
Furthermore, for any $\varepsilon>0$, if we take $N$ large enough, 
\begin{equation*}
\begin{split}
    q_{N+1,k}/q_{N,k} &= \frac{N}{N+1}\frac{\sqrt{(N+1)C+\beta^2}-2\beta^2}{\sqrt{NC+\beta^2}-2\beta^2}e^{\alpha(\sqrt{(N+1)C+\beta^2} - \sqrt{NC+\beta^2})}\\ 
    &\geq e^{-2/N + \frac{\alpha C}{2\sqrt{(N+1)C+\beta^2}}} \\ 
    &\geq e^{\frac{\alpha\sqrt{C}}{2\sqrt{N}}(1-\varepsilon)}.
\end{split}
\end{equation*}
Thus, for $N$ large enough, using the bound $q_{N,k} \leq \frac{2\beta}{\sqrt{NC}}$, $\beta =\alpha^{-1}$,
\begin{equation*}
\begin{split}
  q_{N,k} +    \frac{\theta_N^2S(k)}{2} &\leq q_{N,k}\left(1 + \frac{\alpha^2q_{N,k}\pi^2(1+\varepsilon)}{12}\right) \\
  &\leq q_{N,k}\left(1 + \frac{\pi^2(1+\varepsilon)\alpha}{6\sqrt{NC}}\right) \leq q_{N,k}e^{\frac{\pi^2(1+\varepsilon)\alpha}{6\sqrt{NC}}}\\
& \leq q_{N+1,k}e^{\frac{\pi^2(1+\varepsilon)\alpha}{6\sqrt{NC}}-\frac{\alpha\sqrt{C}}{2\sqrt{N}}(1-\varepsilon) } \leq q_{N+1,k}.
\end{split}
\end{equation*}
For this last inequality to hold, we need:
\begin{equation*}
    \frac{\alpha\sqrt{C}(1-\varepsilon)}{2\sqrt{N}} - \frac{\pi^2(1+\varepsilon)\alpha}{6\sqrt{NC}} \geq 0 \Leftrightarrow C \geq \frac{\pi^2(1+\varepsilon)}{3(1-\varepsilon)},
\end{equation*}
which is true by choosing $\varepsilon$ small enough.
Hence, we obtain 
\begin{equation}\label{fakerecurrence}
 q_{N+1,k}-q_{N,k} \geq     \frac{\theta_N^2S(k)}{2}. 
 \end{equation}

Equation \eqref{fakerecurrence} is the equivalent of \eqref{rec:UB} if all the $q_{N,\ell}$, $1 \leq \ell \leq k$, came from model~$2$. We claim that the inequality is still satisfied if we replace $q_{N,\ell}$ by model $1$ for the values of $\ell$ such that $\log(\ell) < \sqrt{NC+\beta^2} - \beta$. We know that we will not have to switch $ \ell = k$.  

Assume that $N$ is large enough so that $\sqrt{NC+\beta^2} - \beta < \sqrt{NC}$. To compare the two models for $\log(\ell) < \sqrt{NC+\beta^2} - \beta$, we will show that the first model has smaller negative-log-derivative at all $\ell$ such that $\log \ell < \sqrt{NC+\beta^2} - \beta$. This will imply model $2$ is greater than model $1$ for such $\ell$'s. First, for the second model 
\[
  -\frac{\partial }{\partial m}  \log \bigg[ (\frac{2\beta\sqrt{NC+\beta^2}-2\beta^2}{NC})e^{-\beta^{-1}(m - (\sqrt{NC+\beta^2} - \beta))}  \bigg] = \alpha,\]
while  for the first model:
\begin{equation*}
   -\frac{\partial }{\partial m} \log(1-\frac{m^2}{NC}) = \frac{2m}{NC(1-\frac{m^2}{NC})}.
\end{equation*}

We see that this is an increasing function of $m$, and if we evaluate at boundary $m =  \sqrt{NC+\beta^2}-\beta$, we have:
$$\frac{2(\sqrt{NC+\beta^2}-\beta)}{NC - (\sqrt{NC+\beta^2}-\beta)^2} = \frac{2(\sqrt{NC+\beta^2}-\beta)}{2\beta\sqrt{NC+\beta^2}-2\beta^2} = \frac{1}{\beta} = \alpha.$$
Now, we know that the model is continuous, and it is continuously differentiable, and model 1 has higher $\log$ derivative.  This means that model 2 is always greater than model 1 for $\log(\ell) < \sqrt{NC+\beta^2} - \beta$ (and equal at the boundary).  As a result, we will be decreasing $q_{N,\ell}$ when we switch models.  This only further decreases the RHS of \eqref{rec:UB} because of non-negative partial derivatives in $q_{N,\ell}$ for $\ell < k$.  This is because both model 2 and the true values of $q_{N,\ell}$ are monotonic decreasing functions of $\ell$.  
\end{proof}

\begin{proof}[Proof of Proposition \ref{prop:upperbound}]
The proof follows from Lemma \ref{lem:lemadad}, Lemma \ref{lem:middle} and Lemma \ref{lem:tailmodelUB}.
\end{proof}

\section{Lower Bound}
 Similarly to before, our goal is to show that there exists an array $\{q_{N,k}\}_{N\geq 1, k \geq 1}$, and $N_{0}, N_{1} \geq 1$ such that for $N \geq N_0$:
\begin{equation}\label{eq:potato}
q_{N+1,k} - q_{N,k} \leq  \frac{1}{2} \sum_{\ell =1}^{k-1} (q_{N,\ell}-q_{N,\ell+1})(q_{N,k-\ell}-q_{N,k}),
\end{equation}
\begin{equation}\label{eq:probdistcondition}
0 \leq q_{N,k}\leq q_{N,k-1}\leq 1, \text{ for all } k \geq 1,
\end{equation}
and 
\begin{equation}\label{eq:sourcream}
q_{N_{0}+N,k} \leq p_{N_{1}+N,k}, \quad \text{ for all } \quad k \geq 1.
\end{equation}
Such an array will be called a lower bound model.

Define the sequence of constants $\{a_{k}\}$ inductively as follows. Set  $a_{1}=0$, $d_{1}=0$, 
\begin{equation*}
    d_k = ((a_2-a_1)a_{k-1}+(a_3-a_2)a_{k-2}+\cdots+(a_{k-1}-a_{k-2})a_2),
\end{equation*}
and 
\begin{equation}\label{eq:kimjong-un}
    a_k = 1 + \sqrt{1+d_k}.
\end{equation}
Calculated as such, we have $a_2 = 2$, and $a_3 = 1 + \sqrt{5}$, and so on.

\begin{lemma}\label{lem:Iamgettingbetter}
The array $q_{N,k} = 1 - \frac{a_k}{N}$ satisfies \eqref{eq:potato} and for each $k\geq 1$, there is $N_{0}(k)$ and $N_{1}(k)$ so that 
$$q_{N_{0}(k),k} \leq p_{N_{1},k}.$$
\end{lemma}
\begin{proof}
We start by showing that \eqref{eq:potato} is satisfied for large $N$. Note that
\begin{equation*}
    q_{N+1,k}-q_{N,k} = \frac{a_k}{N(N+1)} \leq \frac{a_k}{N^2}.
\end{equation*}
Likewise, we can calculate the right hand side of \eqref{eq:potato} and use the definition of $a_{k}$ 
\begin{equation*}
    a_k^2-2a_k-((a_2-a_1)a_{k-1}+(a_3-a_2)a_{k-2}+\cdots+(a_{k-1}-a_{k-2})a_2) = 0,
\end{equation*}
to obtain
\begin{equation*}
    \bigg(a_k^2-2a_1a_k-((a_2-a_1)a_{k-1}+(a_3-a_2)a_{k-2}+\cdots+(a_{k-1}-a_{k-2})a_2-a_{k-1}a_1)\bigg) \geq 2a_k,
\end{equation*}
which is exactly \eqref{eq:potato} in this case. Inequality \eqref{eq:sourcream} holds since we have a finite number of terms and $p_{N,k}$ goes to $1$ for each fixed $k$. 
\end{proof}

Fix $K \geq 1$ and let $c < \pi^{2}/3$. Let $b_k,1\leq k\leq K$ be any non-negative, monotonic non-decreasing sequence such that $b_K = c^{-1}\log(K)^2$ with 
\begin{equation}\label{eq:bkisaniceconstant}
1-\frac{b_{k}}{N} \leq p_{N+N_{1},k},
\end{equation}
for some $N_{1} \geq 1$.
   Define 
\begin{equation}\label{def:auffingerfamilyvacation}
 q_{N,k} = \begin{cases} 
      1-\frac{b_k}{N} & k < K \\
      1 - \frac{\log(k)^2}{Nc} & k \geq K, \log(k) < \sqrt{Nc} \\ 
      0 & k \geq K, \log(k) \geq \sqrt{Nc}.
   \end{cases}
\end{equation}
Note in particular that $q_{N,k}$ is actually a CDF for all large $N$.  Then, we have the following proposition, which is the main result of this section.
\begin{proposition}\label{prop:adidas} 
There exists $\bar{K}$ large such that if 
\[
q_{N,k} \leq p_{N+N_{1},k}, \text{ for all } K\leq k\leq \bar K  
\]
then $q_{N,k}$ satisfies the inequalities \eqref{eq:potato}, \eqref{eq:probdistcondition} and \eqref{eq:sourcream} for $k \geq \bar{K}$ and for all large $N$.  
\end{proposition}

Before starting the proof of the proposition we will need a few lemmas. 

\begin{definition}
Let $A>0$. We say that inequality \eqref{eq:potato} is satisfied with breathing room $A$ if 
\[A  \leq  \frac{1}{2} \sum_{\ell =1}^{k-1} (q_{N,\ell}-q_{N,\ell+1})(q_{N,k-\ell}-q_{N,k}) - q_{N+1,k} + q_{N,k}.\]
\end{definition}

\begin{lemma}\label{lem:king}
For all $\delta > 0$, there exists $K_\delta$ such that for all $k \geq K_\delta$ and for all $c < \frac{\pi^2}{3}$, when we set $q_{N,j}' = 1 - \frac{\log(j)^2}{Nc}$ for all $N,j$, then inequality \eqref{eq:potato} holds for $q_{N,j}'$ with breathing room $\frac{\epsilon\log(k)^2}{N^2}$, where $\epsilon = \frac{1}{c^2}(\frac{\pi^2}{3}(1-\delta)-c)$.
\end{lemma}
\begin{proof}
Let $\delta > 0$.  First, we have:
\begin{equation*}
    q_{N+1,k}' - q_{N,k}' = \frac{\log(k)^2}{N(N+1)C} \leq \frac{\log(k)^2}{N^2c}.
\end{equation*}
Now, the right side of \eqref{eq:potato} is given by
\begin{equation*}
    X = \frac{1}{2c^2N^2}S,
\end{equation*}
where
\begin{equation*}
\begin{split}
    S :&= \sum_{j=1}^{k-1} (\log(j+1)^2 - \log(j)^2)(\log(k)^2 - \log(k-j)^2) \\
    &\geq \sum_{j=1}^{k-1} (\frac{2\log(j+1)}{j+1})(\log(k)^2 - \log(k-j)^2).
\end{split}
\end{equation*}
Here, we used that $\log(j+1)^2 - \log(j)^2 \geq \frac{2\log(j+1)}{j+1}$ for each $j \geq 3$.  Although this does not hold for smaller $j$, the above bound is valid for large $k$ by same techniques as in the upper bound.

Next, let  $A$ be a large integer so that $\frac{1}{A} < \delta/3$. 
Choose $k$ large enough so that:
\begin{equation}\label{shoehorn}
    \frac{\log(k) - \log A}{\log(k)}\frac{k-A}{k} \geq 1 - \delta/3.
\end{equation}
As 
\[
\log(k)^2-\log(k-j)^2 = -2\log(k)\log\left(\frac{k-j}{k}\right) - \log\bigg(\frac{k-j}{k}\bigg)^2,
\]  
if set 
\begin{equation*}\label{eq:guacamole}
B_k = \sum_{j=1}^{k-1} \bigg(\frac{1}{j}\bigg)\log\bigg(1 - \frac{j}{k}\bigg)^2,
\end{equation*}
we have:
\begin{equation*}
\begin{split}
    S &\geq \sum_{j=1}^{k-1} (\frac{2\log(j+1)}{j+1})\left(-2\log(k)\log\left(\frac{k-j}{k}\right)-\log(\frac{k-j}{k})^2\right) \\ 
    &\geq 4\log(k)^2\sum_{j=\lfloor k/A \rfloor}^{k-1} \frac{\log(j+1)}{\log(k)}\frac{j}{j+1}\bigg(\frac{-1}{j}\bigg)\log\bigg(\frac{k-j}{k}\bigg) - B_k\log(k) \\ 
    &\geq 4\log(k)^2\sum_{j=\lfloor k/A \rfloor}^{k-1} \frac{\log(k) - \log(A)}{\log(k)}\frac{k-A}{k}\bigg(\frac{-1}{j}\bigg)\log\bigg(\frac{k-j}{k}\bigg) - B_k\log(k) \\ &\geq 4(1-\delta/3)\log(k)^2\sum_{j=\lfloor k/A \rfloor}^{k-1} \bigg(\frac{-1}{j}\bigg)\log\bigg(\frac{k-j}{k}\bigg) - B_k\log(k) \\&= 4(1-\delta/3)\log(k)^2M_A(k) - B_k\log(k).
\end{split}
\end{equation*}
In the last line we defined:
\begin{equation}\label{ref:macandcheese}
    M_A(k) = \sum_{j=\lfloor k/A \rfloor}^{k-1} \bigg(\frac{-1}{j}\bigg)\log\bigg(\frac{k-j}{k}\bigg).
\end{equation}
 We choose $k$ large enough and use Lemma \ref{lem:hotdogsandbeans} below to get
\begin{equation}\label{dog}
    B_k\log(k) \leq 12 \log(k) \leq \frac{4\delta\log(k)^2\pi^2}{10}.
\end{equation}
Therefore, we will have by Lemma \ref{lem:cherry} below for large $k \geq K_{\delta}$ that:
\begin{equation*}
    S \geq \frac{2\pi^2(1-\delta)\log(k)^2}{3}.
\end{equation*}
Thus, $X \geq \frac{\pi^2(1-\delta)\log(k)^2}{3c^2N^2}$.  This gives us $\frac{\epsilon \log(k)^2}{N^2}$ breathing room, where \[\epsilon = \frac{1}{c^2}\bigg(\frac{\pi^2}{3}(1-\delta)-c\bigg),\] as desired.  
\end{proof}
We now state and prove the two bounds used in the end of the proof of Lemma~\ref{lem:king}.
\begin{lemma} \label{lem:cherry} Let $M_{A}$ be defined as in \eqref{ref:macandcheese}.  Then
\[ \liminf_{k \to \infty} M_A(k) \geq \frac{\pi^2}{6}-\frac{\pi^2}{6A}.\]
\end{lemma}
\begin{proof}
By Taylor's theorem,
\begin{equation*}
    M_A(k) = -\sum_{j=\lfloor k/A \rfloor}^{k-1} \bigg(\frac{1}{j}\bigg)\log\bigg(1 - \frac{j}{k}\bigg) = \sum_{j=\lfloor k/A \rfloor}^{k-1} \bigg(\frac{1}{j}\bigg)\bigg(\frac{j}{k}+\frac{j^2}{2k^2}+\frac{j^3}{3k^3} + \cdots\bigg).
\end{equation*}
Switching the order of summation, we can write the function of interest using the function $h(k)$ from Lemma \ref{lem:calculuseries}:
\begin{equation*}
    M_A(k) = \sum_{a = 1}^{\infty}\sum_{j=\lfloor k/A \rfloor}^{k-1}\frac{j^{a-1}}{ak^a} = h(k) - \sum_{a = 1}^{\infty}\sum_{j=1}^{\lfloor k/A \rfloor-1}\frac{j^{a-1}}{ak^a}.
\end{equation*}
Notice that we can bound the sum of the $(a-1)$ powers of $j$  as:
\begin{equation*}
    \sum_{j=1}^{\lfloor k/A \rfloor-1}j^{a-1} \leq \int_{1}^{k/A}x^{a-1}dx \leq \frac{(k/A)^a}{a} \leq \frac{k^a}{aA}.
\end{equation*}
Consequently:
\begin{equation*}
    \sum_{a = 1}^{\infty}\sum_{j=1}^{\lfloor k/A \rfloor-1}\frac{j^{a-1}}{ak^a} \leq \sum_{a = 1}^{\infty}\frac{1}{Aa^2} = \frac{\pi^2}{6A}.
\end{equation*}
Combined with Lemma \ref{lem:calculuseries}, we achieve the result.
\end{proof}
\begin{lemma}\label{lem:hotdogsandbeans}
$B_k$ is uniformly bounded above by a constant $\rho < 12$ for all $k$.
\end{lemma}
\begin{proof}

We proceed as in the proof of the lemma above and we write
\begin{equation*}
    B_k = \sum_{a = 1}^{\infty}S_a\sum_{j=1}^{k-1}\frac{j^{a-1}}{k^a}, \quad \text{ with }S_a = \sum_{i=1}^{a-1}\frac{1}{i(a-i)}.
\end{equation*}
We can bound $S_a$ as
\begin{equation*}
    S_a = \sum_{i=1}^{a-1}\frac{1}{i(a-i)} \leq 2\sum_{i=1}^{a}\frac{1}{i(a/2)} = \frac{4H_{a}}{a} \leq \frac{4(\log a+1)}{a}.
\end{equation*}
Consequently:
\begin{equation*}
    B_k \leq \sum_{a = 1}^{\infty}\frac{4(\log a+1)}{a^2} < 12,
\end{equation*}
which ends the proof.
\end{proof}
We are now ready to prove Proposition \ref{prop:adidas}.

\begin{proof}[Proof of Proposition \ref{prop:adidas}]
First, we let $N_0$ be large enough so that that $q_{N,k}$ is a valid CDF and thus satisfies $\eqref{eq:probdistcondition}$. It is also not difficult to see that $q_{N,k}$ satisfies \eqref{eq:sourcream}. We now prove that inequality \eqref{eq:potato} is satisfied for each $k$.  We first consider the case where  $\log k < \sqrt{Nc}$. 

Note that $q_{N,j} \neq q_{N,j}'$ only if $j \leq K$. We only consider $k$ sufficiently large ($k \geq 2K$),  so that the difference $q_{N,j}-q_{N,j}'$ is given by $\frac{e_j}{N}$, where $e_j = b_j - \frac{\log(j)^2}{c}$.

Choose $\bar{K} \geq 2K$ large such that
\begin{equation}\label{JoniMitchell}
    8\sum_{j=1}^{K-1}e_j \leq \epsilon \bar{K}\log(\bar{K}).
\end{equation}

We now calculate the effect of the difference $q_{N,j}-q_{N,j}'$ on the right hand side of \eqref{eq:potato}.
As before, set 
 \begin{eqnarray*}
X&=&  \frac{1}{2} \sum_{\ell =1}^{k-1} (q_{N,\ell}-q_{N,\ell+1})(q_{N,k-\ell}-q_{N,k}), \\
 X' &=& \frac{1}{2} \sum_{\ell =1}^{k-1} (q_{N,\ell}'-q_{N,\ell+1}')(q_{N,k-\ell}'-q_{N,k}').
\end{eqnarray*}

Taking derivative of $X$ with respect to $q_{N,j}$, we get
\begin{equation*}
 q_{N,k-j} - q_{N,k-j+1} \leq \frac{2\log (k-j)}{(k-j)cN}.
\end{equation*}
Thus,
\begin{equation}\label{eq:carhorn}
    |X-X'| = \sum_{j = 1}^{K-1} \frac{e_j}{N}\frac{2\log (k-j)}{(k-j)cN} \leq  4\sum_{j=1}^{K-1}e_j\frac{\log(k)}{kN^2} \leq \frac{\epsilon\log(k)^2}{2N^2}.
\end{equation}
The last inequality holds for $k \geq \bar{K}$.  This shows that inequality \eqref{eq:potato} is satisfied for the case $\bar{K} \leq k$ and $\log(k) < \sqrt{Nc}$ since the bound obtained in \eqref{eq:carhorn} is less than the breathing room from Lemma \ref{lem:king}.  

Now, we consider the case where $\sqrt{Nc} \leq \log(k) < \sqrt{(N+1)c}$.  Then for $j$ in the range $\sqrt{Nc} \leq \log(j) < \sqrt{(N+1)c}$, we need to increase from $q_{N,j}'$ to $0$. 
Repeating the strategy above all partial derivatives are non-negative except for
\begin{equation*}
    \frac{\partial (X'+q_{N,k}')}{\partial q_{N,k}'} = 1 - q_{N,1}' + q_{N,k}' = q_{N,k}' = 1 - \frac{\log(k)^2}{cN} \geq 1 - \frac{(N+1)c}{cN} = -\frac{1}{N}.
\end{equation*}
However, $\frac{1}{N}$ is also the amount that we need to increase $q_{N,k}'$ to $0$ in this range and the cost incurred is bounded above by
\begin{equation*}
    \frac{1}{N^2} \leq \epsilon\frac{Nc}{4N^2} \leq \epsilon\frac{\log(k)^2}{4N^2},
\end{equation*}
which is less than the remaining breathing room. Thus, inequality \eqref{eq:potato} is satisfied for all $k$ so that $ \log (k) \leq \sqrt{(N+1)c}$. In the case $ \log (k) > \sqrt{(N+1)c}$, we do not need to check \eqref{eq:potato} as, by definition, $q_{N+1,k}=0$ is a lower bound for $p_{N+1,k}$.

\end{proof}
\subsection{Construction of the sequence $b_{k}$}
We now construct the sequence that we use in \eqref{def:auffingerfamilyvacation}. In particular we verify that for any $c<\pi^{2}/3$ we can find a sequence $b_{k}$ and $K\geq 1$ so that \eqref{eq:bkisaniceconstant} holds and $b_K = c^{-1}\log(K)^2$. 

We start by the following Lemma. 
\begin{lemma}
There exist integers $K_0$ and $N_{1}$ such that for all $k \geq K_0$, \[1 - \frac{\log(k)^2}{N} \leq p_{N+N_{1},k} \quad \text{ for all  } N \geq 1.\]
\end{lemma}
\begin{proof}
As shown in Proposition \ref{prop:adidas}, this follows by showing that $b_k = \log(k)^2$ for $K \leq k \leq \bar{K}$ is a lower bound for $p_{N,k}$. Lets first  calculate how large the constant $\bar{K}$ from Proposition \ref{prop:adidas} needs to be in this case.  We need to check \eqref{shoehorn},\eqref{dog}, \eqref{JoniMitchell}, and that $\bar{K} \geq K_\delta$.  Set $\delta = 2/3$, $A = 8$, $K = 33$. For $k < 33$ we set $b_{k}=a_{k}$, from \eqref{eq:kimjong-un}.    For these choices, Equation \eqref{dog} is immediately satisfied for $k \geq 100$ and  $\epsilon \geq \frac{1}{12}$.  Equation \eqref{JoniMitchell} is verified numerically for $k \geq 1000$, while \eqref{shoehorn} holds for $k \geq 12000$.  In order for $\bar{K} \geq K_{\delta}$, we need $h(k) \geq \pi^2/6(1-1/11)$ for all  $k \geq \bar{K}$. This holds because 
\[
\frac{\pi^2}{6A} + \frac{\pi^2}{6\cdot 11} \leq \frac{\pi^2\delta/3}{6}
\]  and that for $k \geq 10000$, $f(k) \geq 1.50 \geq \pi^2/6(1-1/11)$. Indeed,
\begin{equation*}
\begin{split}
    h(k) = \sum_{a = 1}^{\infty}\sum_{j=1}^{k-1}\frac{j^{a-1}}{ak^a} &\geq \sum_{a = 1}^{7}\sum_{j=1}^{k-1}\frac{j^{a-1}}{ak^a} \\
    &= \frac{600-7826k^2+57435k^4-316890k^5+266681k^6}{176400k^6} \\
    &\geq 1.51 - \frac{382751}{176400k} \geq 1.5.
\end{split}
\end{equation*}
Thus, we set $\bar{K} = 12000$ and check that for $K \leq k \leq \bar{K}$ that $a_k \leq \log(k)^2$.
\end{proof}

\subsubsection{Choosing of the constants}
We now fix the constants that we use in the remaining of this section. Choose $\delta > 0$ such that $\pi^2/3(1-\delta) > c$.
Then set $\tilde{c} = \frac{\pi^2}{3}(1-\delta)$. Next, recall the constants $a_{j}$ from \eqref{eq:kimjong-un} and let $K_{0}=33$.

Set $\triangle_{0}=\tilde{c}-1$, $C_1 = \sum_{j=1}^{K_0-1}4a_j$, and $K_r = \lfloor{e^{\sqrt{r+r_0}}}\rfloor$ with $r_0$ large enough so that 

\[ K_1 \geq K_{\delta}, \; \frac{C_1\sqrt{2+r_0}}{e^{2+r_0}-1} \leq \frac{\triangle_0}{2\tilde{c}^2}, \text{ and } K_2 - K_1 > K_0.
\]

Next, set 
\begin{eqnarray*}
\quad C_2 &=& 3+r_{0}, \quad \quad \quad C_3 = \sqrt{1+r_0}, \\
C_4 &=& \sqrt{1+r_0}\frac{4C_3\pi^2}{3}, \quad   C_5 = C_2 + 2C_4, \\
C_6 &=& 2r_0\triangle_{0}+12, \quad C_{total} = C_6 + C_1 + 2C_5. 
\end{eqnarray*}

Let $a$ be so that
$a \leq 1$, $a \leq \frac{1}{2\tilde{c}C_{total}}$, $a \leq \frac{1}{4K_{0}\tilde{c}^{2}}$ and $b$ chosen small enough so that
$b \leq \frac{a}{2(r_0+3)}.$
Last, define
\begin{equation*}
     \triangle_r = \begin{cases}
     \tilde{c} - 1, & r=1 \\
     \triangle_1r^{-b}, & r > 1,
     \end{cases}
\end{equation*}
 $c_r = \tilde{c} - \triangle_r$ and define
\begin{equation*}
q_{N,k} = 
\begin{cases}
1 - \frac{a_k}{N} & k < K_0, \\
1 - \frac{\log(k)^2}{c_iN} & K_i \leq k < K_{i+1}.
\end{cases}
\end{equation*}
The intuition is that we are constructing a sequence of step functions that are valid lower bound models. These models will have jumps at $k = K_r$ of size $\delta_r$ to $c_r$. 
\begin{proposition}
For each $k$, there exists $N_k$ such that $q_{N,k}$ satisfies \eqref{eq:potato} for $N \geq N_k$.
\end{proposition}
\begin{proof}
$N_k$ is first chosen large enough so that $\{ q_{N,j}\}_{j\geq 1}$ is a valid CDF up to $j=k$.  For $k < K_0$, we refer the reader to our exact calculation of $a_k$ in \eqref{eq:kimjong-un}.  For $K_0 \leq k < K_2$, we refer to Lemma \ref{lem:king} above.  Next, define
\begin{equation*}\begin{split}
    \delta_{r+1}= \triangle_{r+1} - \triangle_r = \triangle_1(r^{-b} - (r+1)^{-b}) \leq \triangle_1b r^{-(1+b)} &= \frac{b\triangle_r}{r} 
    \\ &\leq \frac{a\triangle_{r+1}}{r+r_0+2} \leq \frac{a\triangle_{r+1}}{\log(K_{r+1})^2}.
\end{split}
\end{equation*}

Fix $r \geq 1$ and let $K_{r+1}~\leq k <~K_{r+2}$.  As in the proof of Proposition \ref{prop:adidas}, we will have to estimate the errors when we switch models. We change the variables from model $r+1$ to model 1 by interpolating the $c$ values continuously from $c_{r+1}$ to $c_r$, and then to $c_{r-1}$, all the way to $1$.  Precisely, we choose non-increasing functions $c_i(t)$ from $c_{r+1}$ to $c_i$ as follows 
\begin{equation*}
    c_i(0) = c_{r+1}, c_i(1) = c_{i}, |c_i(t) - c_{i+1}(t)| \leq \delta_{i+1}
\end{equation*}
and we set
\begin{equation*}
    q_{N,m}(t) = 1 - \frac{\log(m)^2}{c_i(t)N}, \text{ for } K_i \leq m < K_{i+1}.
\end{equation*}
When all $q_{N,j}$ are equal to $1 - \frac{\log(j)^2}{Nc_{r+1}}$, we can  apply Lemma \ref{lem:king} with $c=c_{r+1}$ to obtain that the breathing room for this model is equal to
\begin{equation*}
     \epsilon_{r+1} = \frac{1}{c_{r+1}^2}(\frac{\pi^2}{3}(1-\delta) - c_{r+1}) \geq \frac{1}{\tilde{c}^2}\triangle_{r+1}.
\end{equation*}
 As a result, we have:
\begin{equation*}
    \frac{\epsilon_{r+1}\log(k)^2}{N^2} \geq \frac{(r+1)\triangle_{r+1}}{N^2\tilde{c}^2} \geq \frac{r\triangle_{r}}{N^2\tilde{c}^2}.
\end{equation*}
To shorten the notation, let $f(t)=f(q_{N,1}(t), \ldots, q_{N,k}(t))$, where $f$ is defined in \eqref{def:mainfunction}. Now

\begin{equation*}
f(0) - f(1) \leq \sum_{j = 1}^{k} (q_{N,j}(0)-q_{N,j}(1))\sup_{t\in (0,1)} \frac{\partial f(q_{N,1}(t), q_{N,2}(t),\dots,q_{N,k}(t))}{\partial q_{N,j}} :=E(r).
\end{equation*}
The proof of the proposition will be complete if we show the following claim.
\begin{claim} For any $r \geq 1$,
\begin{equation*}
    E(r) \leq \frac{r\triangle_{r}}{N^2\tilde{c}^2}.
\end{equation*}
\end{claim}
We start by
\begin{proposition} \label{prop:pain}
The following holds for $1 \leq j <k$
\begin{equation}\label{eq:bounds.}
\begin{split}
    \sup_{t\in (0,1)}\frac{\partial f}{\partial q_{N,j}}(t) &= \sup_{t\in (0,1)} q_{N,k-j}(t) - q_{N,k-j+1}(t)  \\
        &\leq 
    \begin{cases} 
      \frac{1}{N}, & K_{i}\leq k-j<K_{i+1}-1 \text{ and }  k-j < 3, 
      \cr \frac{2\log(k-j)}{(k-j)N}, & K_{i} \leq k-j<K_{i+1}-1 \text{ and } k-j \geq 3.
   \end{cases}
\end{split}
\end{equation}
\end{proposition}
\begin{proof}
Let $m = k-j$.  We have
\begin{equation*}
    q_{N,m}(t) = 1 - \frac{\log(m)^2}{c_i(t)N}, \quad q_{N,m+1}(t) = 1 - \frac{\log(m+1)^2}{c_i(t)N}.
\end{equation*}
And we can bound their difference:
\begin{equation*}
    q_{N,m}(t) - q_{N,m+1}(t) = \frac{1}{c_i(t)N}\bigg(\log(m+1)^2 - \log(m)^2\bigg) \leq \frac{1}{N}\bigg(\log(m+1)^2 - \log(m)^2\bigg).
\end{equation*}
The bound \eqref{eq:bounds.} now follows as $\log(m+1)^2 - \log(m)^2 \leq \frac{2\log(m)}{m}$ if $m \geq 3$ and $\log(m+1)^2 - \log(m)^2 \leq 1$ is $m = 1,2$.  
\end{proof}
\begin{proof}
We estimate the differences $q_{N,j}(0) - q_{N,j}(1)$. If $K_{\alpha_{j}}\leq k-j<K_{\alpha_{j}+1}-1$, $\alpha_{j} \geq 0$, then
\begin{equation*}
    q_{N,j}(0) - q_{N,j}(1) = (1 - \frac{\log(j)^2}{c_{r+1}N}) - (1 - \frac{\log(j)^2}{c_{\alpha_{j}}N}) \leq \frac{\log(j)^2(c_{r+1} - c_{\alpha_{j}})}{N}. 
\end{equation*}
By Propostion \ref{prop:pain} and the above display, if we set
\[ E_{1}=  \sum_{j=k-K_{0}+1}^{k}  (q_{N,j}(0)-q_{N,j}(1))\sup_{t\in (0,1)} \frac{\partial f(q_{N,1}(t), q_{N,2}(t),\dots,q_{N,k}(t))}{\partial q_{N,j}},
\]
then
\begin{equation*}
\begin{split}
    E(r) &\leq \sum_{j = 1}^{k-K_{0}}\frac{\log(j)^2(c_{r+1} - c_{\alpha_j})}{N^2}\bigg(\log(k-j+1)^2 - \log(k-j)^2\bigg) + E_{1}\\
    &\leq \sum_{j = 1}^{k-K_{0}}\frac{\log(j)^2(\sum_{i=\alpha_j+1}^{r+1}\delta_i)}{N^2}\bigg(\log(k-j+1)^2 - \log(k-j)^2\bigg) + E_{1} \\
&\leq \sum_{i = 1}^{r+1}\delta_i\sum_{j=1}^{\min(K_i, k-K_{0})}\frac{\log(j)^2}{N^2}\bigg(\log(k-j+1)^2 - \log(k-j)^2\bigg) + E_{1} \\ 
&\leq \sum_{i = 1}^{r+1}\delta_i\sum_{j=1}^{\min(K_i, k-K_{0})}\frac{\log(K_i)^2}{N^2}\bigg(\log(k-j+1)^2 - \log(k-j)^2\bigg) + E_{1}\\ 
&= \sum_{i = 1}^{r+1}\delta_i\frac{\log(\min(K_i, k-K_{0}))^2}{N^2}\bigg(\log(k)^2 - \log(k-\min(K_i, k-K_{0}))^2\bigg) + E_{1} \\
&\leq \sum_{i = 1}^{r+1}\frac{a\triangle_i}{N^2}\bigg(\log(k)^2 - \log(k-\min(K_i, k-K_{0}))^2\bigg) + E_{1} \\ 
&\leq \sum_{i = 1}^{r}\frac{a\triangle_i}{N^2}\bigg(\log(K_{r+1})^2 - \log(K_{r+1}-\min(K_i, k-K_{0}))^2\bigg)+ \frac{aC_2\triangle_rr}{N^2} +E_{1}.
\end{split}
\end{equation*}
The last inequality follows because differences are largest for $k$ smaller, and we have $k~\geq~K_{r+1}$.  Now, observe that 
\begin{equation*}
\begin{split}
    \sum_{i = 1}^{r}\frac{a\triangle_i}{N^2}&\bigg(\log(K_{r+1})^2 - \log(K_{r+1}-\min(K_i, k-K_{0}))^2\bigg) \\
    &\quad \leq \sum_{i = 1}^{r}\frac{a\triangle_i}{N^2}\bigg(\log(K_{r+1})^2 - \log(K_{r+1}-K_i)^2\bigg) .
  \end{split}
\end{equation*}
Combining the last display, Lemma \ref{lem:C5} below, and the sequence of inequalities above we arrive at 
\begin{equation}\label{eq:waitingforTuca}
 E(r) \leq  \frac{aC_5\triangle_rr}{N^2} + E_{1}.
\end{equation}

\begin{lemma}\label{lem:C5}
Set $K_n = \lfloor e^{\sqrt{r+r_0}}\rfloor$.  Then $S(r) = \sum_{m = 1}^{r-1}\frac{\triangle_m}{\triangle_r}(\log(K_r)^2 - \log(K_{r}-K_m)^2)~\leq~C_4r$ for all $r$.
\end{lemma}

\begin{proof}
First, we can take $K_r = e^{\sqrt{r+r_0}}$ as this will multiply each term in the summation by a factor of $2$.  Using the fact that $b<1/2$, we obtain
\begin{equation*}
    S(r) \leq \sum_{m = 1}^{r-1}\frac{\sqrt{r}}{\sqrt{m}}(\log(K_r)^2 - (\log(K_r)+\log(1-\frac{K_m}{K_r}))^2) \leq -2\sum_{m = 1}^{r-1}\frac{\sqrt{r}}{\sqrt{m}}\log(K_r)\log(1-\frac{K_m}{K_r}).
\end{equation*}
Now, since $\frac{K_m}{K_r} < 1$, we can expand using $\log$ power series.
\begin{equation*}
    S(r) \leq 2\log(K_r)\sum_{m=1}^{r-1}\sum_{j=1}^{\infty}\frac{1}{j}\frac{\sqrt{r}}{\sqrt{m}}\bigg(\frac{K_m}{K_r}\bigg)^j.
\end{equation*}
Since all terms are positive, we can switch the order of summation:
\begin{equation*}
    S(r) \leq 2\log(K_r)\sum_{j=1}^{\infty}\frac{1}{j}\sum_{m=1}^{r-1}\frac{\sqrt{r}}{\sqrt{m}}\bigg(\frac{K_m}{K_r}\bigg)^j.
\end{equation*}
Notice that
\begin{equation*}
\begin{split}
    \sum_{m=1}^{r-1}\frac{\sqrt{r}}{\sqrt{m}}\bigg(\frac{K_m}{K_r}\bigg)^j &= \sqrt{r}e^{-j\sqrt{r+r_0}}\sum_{m=1}^{r-1}\frac{1}{\sqrt{m}}e^{j\sqrt{m+r_0}} \\ 
    &\leq  C_3\sqrt{r+r_0}e^{-j\sqrt{r+r_0}}\sum_{m=1}^{r-1}\frac{1}{\sqrt{m+r_0}}e^{j\sqrt{m+r_0}} \\
    &\leq C_3\sqrt{r+r_0}e^{-j\sqrt{r+r_0}}\int_{r_0}^{r+r_0}\frac{1}{\sqrt{x}}e^{j\sqrt{x}}dx \\ &= C_3\sqrt{r+r_0}e^{-j\sqrt{r+r_0}}\frac{2}{j}\int_{0}^{j\sqrt{r+r_0}}e^{y}dy  \leq C_3\sqrt{r+r_0}e^{-j\sqrt{r+r_0}}\frac{2}{j}e^{j\sqrt{r+r_0}} \\
    &= C_3\frac{2\sqrt{r+r_0}}{j}.
\end{split}
\end{equation*}
In total we have:
\begin{equation*}
    S(r) \leq 2\sqrt{r}\sum_{j=1}^{\infty}\frac{2C_3\sqrt{r+r_0}}{j^2} \leq C_4r,
\end{equation*}
which concludes the proof.
\end{proof}

It remains to bound $E_{1}$. From Proposition \ref{prop:pain}, 
\begin{equation}\label{eq:justwenttothebathroom}
\begin{split}
    E_1 \leq \sum_{j = k-K_{0}+1}^{k} \frac{\log(j)^2(c_{r+1} - c_{r})}{N^{2}} &\leq  \frac{K_{0}\log(K_{r+2})^{2} \delta_{r+1}}{N^2} 
   \\ &\leq \frac{a \triangle_{r+1}K_{0} \log(K_{r+2})^{2}}{N^{2}\log(K_{r+1})^{2}} \leq \frac{2aK_{0}\triangle_0}{N^2} \leq \frac{r\triangle_{r}}{2N^2\tilde{c}^2},
\end{split}
\end{equation}
where the last inequality is due to our choice of $a$. 
Combining \eqref{eq:waitingforTuca} with \eqref{eq:justwenttothebathroom}  we proved the Claim.
\end{proof}

To complete the proof of the lower bound, we still need to verify that $q_{N+1,k}~\leq~p_{N+1+N_{1},k}$. Although $q_{N,j}$ satisfies the recurrence, they are not monotonically non-increasing.  As a result, we define $q_{N,j}'$ to be a monotonically non-increasing modification of $q_{N,j}$: 
\begin{equation*}
q_{N,j}' = 
\begin{cases}
&\max \bigg \{ 1 - \frac{\log(j)^{2}}{Nc_{i}}, 1- \frac{\log(K_{i+1})^{2}}{Nc_{i+1}}\bigg\}, \quad  K_{i} \leq j < K_{i+1}, \text{ for } 1 \leq i \leq r,\\
&q_{N,j}, \quad \text{ if } j\geq K_{i+1} \text{ or } j \leq K_{1}.
\end{cases}
\end{equation*}
It is clear that $q_{N,j}' \geq q_{N,j}$ and now we claim that $q_{N,j}' \leq p_{N+N_{1},j}$. Indeed, as $q_{N,j}~\leq~p_{N+N_{1},j}$ and $q_{N,j}' \neq q_{N,j}$ only if $q_{N,j}' = q_{N,K_{i+1}}$, we have $q_{N,j}' \leq p_{N+N_{1},K_{i+1}} \leq p_{N+N_{1},j}$ by monotonicity of $p$. Lastly, 

\begin{lemma}
The sequence $q_{N,j}'$ is monotonically non-increasing in $j$. 
\end{lemma}
\begin{proof}
It suffices to check the case when $ K_{i} \leq j < K_{i+1}$. Note that by our choice of constants, 
\begin{equation*}
    \delta_{i+1} \leq \frac{a\triangle_i}{(i+r_0+1)} \leq \frac{a}{i+r_0+1} \leq \frac{1/2}{i+r_0+1/2},
\end{equation*}
which implies 
\begin{equation}\label{eq:Sunday!!}
    \frac{c_i}{c_{i+1}} = \frac{c_{i+1} - \delta_{i+1}}{c_{i+1}} \geq 1 - \delta_{i+1} \geq 1 - \frac{1/2}{i+r_0+1/2}.
\end{equation}
On the other hand, 
\begin{equation*}
    \frac{\log(K_i)^2}{\log(K_{i+1})^2} \leq \frac{i + r_0}{i + r_0 + 1 - 1/2} = 1 - \frac{1/2}{i+r_0+1/2},
\end{equation*}
which combined with \eqref{eq:Sunday!!} leads to
\begin{equation*}
   \frac{\log(K_i)^2}{\log(K_{i+1})^2} \leq \frac{c_i}{c_{i+1}}.
\end{equation*}
This concludes that \[\frac{\log(K_i)^2}{Nc_{i}} \leq \frac{\log(K_{i+1})^2}{Nc_{i+1}},\]
which ends the proof of the lemma. 
\end{proof} 
Furthermore, the size of this increase when $q_{N,j} \neq q_{N,j}'$ is bounded above by:
\begin{equation*}
  |q_{N,j}'-q_{N,j}| =   \frac{\log(j)^2}{Nc_{i}} - \frac{\log(K_{i+1})^2}{Nc_{i+1}}  \leq \frac{\log(j)^2}{Nc_{i}} - \frac{\log(j)^2}{Nc_{i+1}} \leq \delta_{i+1}\frac{\log(j)^2}{N^2} \leq \frac{a\triangle_{i+1}}{N}. 
\end{equation*}
So, as before, we estimate the difference in $f$ as
\begin{equation} \label{eq:itasarwrmoinruiogherwyugweruygewryugryugeryukfg56t32342}
\begin{split}
    E_c' := f(q_{N,1}, \ldots, q_{N,k})-f(q_{N,1}',\ldots, q_{N,k}')&\leq \sum_{j = 1, K_i = k-j+1}^{k}|q_{N,j}'-q_{N,j}|\frac{\log(k-j)^2\delta_i}{N}\\
    &\leq \sum_{j = 1, K_i = k-j+1}^{k}|q_{N,j}'-q_{N,j}|\frac{a\triangle_{i}}{N} \\
    &\leq \sum_{j = 1, K_i = k-j+1}^{k}\frac{a\triangle_{\alpha_j+1}}{N} \frac{a\triangle_{i}}{N},
    \end{split}
\end{equation}
where, $\alpha_{j} \geq 1$ is the index such that
\begin{equation*}
    j = k - K_i + 1, \quad K_{\alpha_j} \leq j < K_{\alpha_j+1}.
\end{equation*}

Here we have used the following fact.
\begin{lemma} If  $K_i = k-j+1$, for some  $0\leq i \leq r$, then
\begin{equation*}
    -\frac{\partial f}{\partial q_{N,j}} \leq \frac{\log(k-j)^2\delta_{i}}{N}.
\end{equation*}
\end{lemma}
\begin{proof}
Let $m = k-j$. Note that the negative derivative is bounded by
\begin{equation*}
\begin{split}
    q_{N,m+1} - q_{N,m} &= \frac{1}{c_{i-1}N}\log(m)^2 - \frac{1}{c_iN}\log(m+1)^2 
    \\ &\leq  \frac{\log(m)^2}{N}\frac{\bigg(c_i - c_{i-1}\bigg)}{c_{i-1}^2} \leq  \log(m)^2\frac{\delta_i}{N}.
\end{split}
\end{equation*}
\end{proof}

We bound the sum in \eqref{eq:itasarwrmoinruiogherwyugweruygewryugryugeryukfg56t32342} term by term. Set $m = \max\{i,\alpha_j+1\}$. Note that
\begin{multline*}
    K_{r+1} \leq k+1 \leq K_i + K_{\alpha_j+1} \implies \sqrt{r+r_0} \leq \log(e^{\sqrt{r_0+i}} + e^{\sqrt{r_0+\alpha_j+1}}) \\ \leq \log(e^{\sqrt{r_0+m}}) + \log(2) \leq \sqrt{r_0+m} + \log(2) \implies 1/2(r+r_0) \leq m + r_0 \implies m \geq 1/2(r-r_0).
\end{multline*}
It follows that we can bound:
\begin{equation*}
    \triangle_{\alpha_j+1}\triangle_i \leq  \triangle_0\triangle_{m} \leq 
    \begin{cases} 
      \triangle_0^{2}(1/2(r-r_0))^{-b} \leq \triangle_0^{2}4^br^{-b} = \triangle_{0}\triangle_r4^b & r \geq 2r_0,
      \cr  \triangle_0^{2} & r < 2r_0.
   \end{cases}
\end{equation*}
As a result, using that $\triangle_{0}a4^{b}\leq C_{6}$ when $ r \geq 2r_0$ and that $\triangle_{0} \leq r \triangle_{r}$ when $r < 2r_0$, we obtain $E_c' \leq \frac{aC_6\triangle_rr}{N^2}$.

In total, our error is $\frac{aC_{total}\triangle_rr}{N^2}$, and we require this to be less than $\frac{r\triangle_{r}}{2\tilde{c}^2n^2}$.  
\end{proof}

\section{Proof of Theorem \ref{thm:distribution}}
We now use the estimates obtained in the past two sections to prove Theorem \ref{thm:distribution}.
\begin{proof}[Proof of Theorem \ref{thm:distribution}]
Define $c= \frac{\pi^{2}}{3}.$ It suffices to show that for any  $a \geq 0$,
\begin{equation*}
    \lim_{N \to \infty}\mathbb{P}\left(\frac{\log(X_N)}{\sqrt{cN}} \leq a\right) = \min(a^2,1).
\end{equation*}
If $a = 0$ then by Lemma \ref{lem:Iamgettingbetter} and the fact that $a_{2}=2$ we have for all $N\geq1$,  
$$\mathbb{P}\left(\frac{\log(X_N)}{\sqrt{cN}} \leq a\right) =  \mathbb{P}\left(X_N \leq 1\right) = 1 - p_{N,2} \leq \frac{2}{N-N_{1}}.$$
Consider $0<a<1$. We write
\begin{equation*}
    \mathbb{P}\left(\frac{\log(X_N)}{\sqrt{cN}} \leq a\right) = \mathbb{P} \left(X_N \leq e^{a\sqrt{cN}} \right) = 1 - \mathbb{P}\left(X_N \geq e^{a\sqrt{cN}}\right)=1-p_{N,\lceil e^{a\sqrt{cN}}\rceil}.
\end{equation*}
As $0<a<1$, for each $\delta>0$, we can choose $\tilde{c}$ with $c<\tilde{c} < (1+\delta)c$ such that for large $N$, 
\[ 
a\sqrt{cN} < \sqrt{\tilde{c}N + \beta^2} -\beta,
\] 
where $\beta$ is determined in Proposition \ref{prop:upperbound}. Thus, by Proposition \ref{prop:upperbound}:
\begin{equation}\label{eq:timetogohome1}
\mathbb{P}\bigg(X_N \geq e^{a\sqrt{cN}}\bigg) \leq 1-\frac{a^2cN}{c(1+\delta)(N+N_0)} = 1 - \frac{a^2}{(1+\delta)^{2}}.
\end{equation}
For the lower bound, we choose $\tilde c$ with $a\sqrt{c} < \sqrt{\tilde{c}} < \sqrt{c}$, which implies that $a\sqrt{cN} \leq \sqrt{\tilde{c}N}$, and $a\sqrt{cN}\geq \bar K$ for large $N$, where $\bar K$ comes from Proposition \ref{prop:adidas}. Then, for each $\delta>0$, for large $N$, the following holds by Proposition \ref{prop:adidas}:
\begin{equation}\label{eq:timetogohome2}
\mathbb{P}\bigg(X_N \geq e^{a\sqrt{cN}}\bigg) \geq 1-\frac{a^2cN+1}{c(1-\delta)(N-N_{1})} = 1 - \frac{a^2}{(1-\delta)^{2}}.
\end{equation}
Combining \eqref{eq:timetogohome1} and \eqref{eq:timetogohome2} we have that for each $\delta >0$, if we take $N$ large enough:
\begin{equation*}
    \frac{a^2}{(1+\delta)^{2}} \leq \mathbb{P}\bigg(\frac{\log(X_N)}{\sqrt{cN}} \leq a\bigg) \leq \frac{a^2}{(1-\delta)^{2}}.
\end{equation*}
Since this holds for each $\delta$, we conclude that
\begin{equation*}
    \lim_{N \to \infty}\mathbb{P}\bigg(\frac{\log(X_N)}{\sqrt{cN}} \leq a\bigg) = a^2,
\end{equation*}
as desired.
 
Now take $a>1$ and choose $ \tilde{c}$ so that $a\sqrt{c} > \tilde{c} > c$.  Again, by Proposition \ref{prop:upperbound}, we take $a\sqrt{cN} -1 \geq \sqrt{\tilde{c}(N+N_{0}) + \beta^2}-\beta$, so that for large $N$:
\begin{equation*}
\mathbb{P}\bigg(X_N \geq e^{a\sqrt{cN}}\bigg) \leq \frac{2\beta}{\sqrt{\tilde{c}(N+N_{0})}}.
\end{equation*}
This bound goes to 0 as $N \to \infty$, so:
\begin{equation*}
    \lim_{N \to \infty}\mathbb{P}\bigg(\frac{\log(X_N)}{\sqrt{cN}} \leq a\bigg) = 1
\end{equation*}
Last, we consider the case where $a=1$. We use monotonicity of the CDF and the discussion above to conclude that:
\begin{equation*}
    \lim_{N \to \infty}\mathbb{P}\bigg(\frac{\log(X_N)}{\sqrt{cN}} \leq a\bigg) = 1.
\end{equation*}
This ends the proof of the theorem.
\end{proof}  
We finish this section with the proof of Corollary \ref{thm:expectation}. 
\begin{proof}[Proof of Corollary \ref{thm:expectation}]
It suffices to show that the sequence $\frac{\log X_{N}}{\sqrt{cN}}$ is uniformly integrable. Let $\tilde{c} > \frac{\pi^2}{3}$.  By Proposition \ref{prop:upperbound}, there exists $\beta > 0$ and $N_0$ such that \eqref{cownunchuckpudding} holds. For large $N$:
\begin{equation*}
\begin{split}
   \mathbb{E}&\bigg(\frac{\log X_{N}}{\sqrt{cN}}\bigg)^2\\ 
   &\leq \frac{1}{cN}\bigg(\log(e^{\sqrt{\tilde{c}N+\beta^2}-\beta+1})^2+\sum_{\ell=1}^\infty \mathbb{P}[X_N \geq e^{\sqrt{\tilde{c}N+\beta^2}-\beta+\ell}]\log(e^{\sqrt{\tilde{c}N+\beta^2}-\beta+\ell+1})^2\bigg)
   \\ &\leq \frac{1}{cN}\bigg(2\tilde{c}N+\sum_{\ell=1}^\infty e^{-\beta^{-1}\ell}(\sqrt{\tilde{c}N}+\ell+1)^2\bigg)
   \\ &\leq \frac{1}{cN}\bigg(2\tilde{c}N+\tilde{c}N\sum_{\ell=1}^\infty e^{-\beta^{-1}\ell}(\ell+1)^2\bigg) = \frac{1}{c}\bigg(2\tilde{c}+\tilde{c}\sum_{\ell=1}^\infty e^{-\beta^{-1}\ell}(\ell+1)^2\bigg) < \infty.
\end{split}
\end{equation*}
Since we have bounded the second moments (for all but finitely many $N$) by a constant that does not depend on $N$, we conclude that $\frac{\log(X_N)}{\sqrt{cN}}$ is uniformly integrable. This, along with Theorem \ref{thm:distribution}, implies that \cite{Bill}:
\begin{equation*}
    \lim_{N}\mathbb{E}\bigg[\frac{\log(X_N)}{\sqrt{cN}}\bigg] = \mathbb{E}\left[\lim_{N}\frac{\log(X_N)}{\sqrt{cN}}\right] = \frac{2}{3}.
\end{equation*}
\end{proof}

\section{Other Cases and some questions}\label{sec:othercases}
\subsection{Case $p\neq 1/2$.}
This case was also considered in a slightly different setting in Section $3$ of \cite{Hambly}.
We first examine the case $p < 1/2$,  where $p$ is the probability of placing $+$ at each node.  We obtain a slightly different recurrence for this problem:
\begin{multline*}
    p_{N+1,k} = p_{N,k}^2 + p[2p_{N,k}(1-p_{N,k}) + (p_{N,1} - p_{N,2})(p_{N,k-1}-p_{N,k}) +  \\ + \cdots +(p_{N,k-1} - p_{N,k})(p_{N,1}-p_{N,k})]
\end{multline*}
Rearranging:
\begin{equation}\label{eq:MONDAY!!!!}\begin{split}
p_{N+1,k} &  = 2pp_{N,k}+ (1-p)p_{N,k}^2 + p[(p_{N,1} - p_{N,2})(p_{N,k-1}-p_{N,k}) + \ldots 
\\ &\quad +(p_{N,k-1} - p_{N,k})(p_{N,1}-p_{N,k})].
\end{split}
\end{equation}
Fortunately, \eqref{eq:MONDAY!!!!} has non-negative partial derivatives, for example,
\begin{equation*}
\frac{\partial p_{N+1,k}}{\partial p_{N,k}} = 2p + (1-2p)p_{N,k} + p(2(p_{N,k}-1)) = p_{N,k} \geq 0.
\end{equation*}
This will lead to the following result:
\begin{theorem}\label{thm:ASDpsdoihgui9fer}
If $p<1/2$ then the sequence $X_{N}$ converges in distribution to some non-trivial random variable. 
\end{theorem}
This will follow from the following claim.
\begin{proposition}\label{prop.dsadda}
The sequence $X_{N}$ is tight, that is, for each $\epsilon > 0$, there exists $k$ such that: $\mathbb{P}[X_N \geq k] < \epsilon$ for large $N$.
\end{proposition}

\begin{proof}[Proof of Theorem \ref{thm:ASDpsdoihgui9fer}] By Proposition \ref{prop.dsadda} the sequence of random variables $X_{N}$ is tight. On the other hand, we note that $p_{N,k}=\mathbb P(X_{N}\geq k)$ is monotonic increasing in $N$ for each fixed $k$, thus it converges to a limit. This limit must be non-trivial as $X_{N}\geq 1$ and $0<\lim_{N\to \infty}p_{N,2}<1$ by Lemma \ref{fskdo}. The claimed monotonicity follows inductively using non-negative partial derivatives. 
\end{proof}
\begin{proof}[Proof of Proposition \ref{prop.dsadda}]
Let $c_k =  \lim_{n \to \infty} p_{N,k}$. If we observe \eqref{eq:MONDAY!!!!}, then because each term is approaching a limit, by continuity, we can substitute $c_j$ in for each term ($c_1 = 1$):
\begin{equation*}
    c_k = 2pc_k+ (1-p)c_k^2 + p[(c_1 - c_2)(c_{k-1}-c_k) +\\ \cdots +(c_{k-2} - c_{k-1})(c_2-c_k) + (c_{k-1} - c_k)(c_1-c_k)]
\end{equation*}
Rearranging, we get
\begin{equation}\label{eq:Steveisloud}
    c_k - (1-p)c_k^2 = p[(c_1 - c_2)c_{k-1} + (c_2 - c_3)c_{k-2} +\\ \cdots +(c_{k-2} - c_{k-1})c_2 + c_{k-1}].
\end{equation}
Our goal is to show that for $p < 1/2$, we have $c_k \to 0$ as $k \to \infty$. 
\begin{lemma}\label{fskdo} One has
$c_2 = \frac{p}{1-p} < 1.$
\end{lemma}
\begin{proof}
We have the recurrence for $x_n = \mathbb{P}[X_n=1]$:
\begin{equation*}
    x_{n+1} = (1-p)(1 - (1 -x_n)^2)
\end{equation*}
The fixed points of this recurrence are the solutions to:
\begin{equation*}
    x = (1-p)(1 - (1 -x)^2) \Leftrightarrow x = (1-p)(2x - x^2) \Leftrightarrow x = 0, 1 - \frac{p}{1-p}
\end{equation*}
If we plot the recurrence, then it is clear that $x = 1 - \frac{p}{1-p}$ is an attractor for starting point $x_1 = 1$. The desired $c_2$ follows.
\end{proof}
Since $c_k$ is monotonic non-increasing, we can set $c = \lim_{k \to \infty}c_k$. Then, if we take the limit of both sides of equation \eqref{eq:Steveisloud}, we obtain on the left side 
\begin{equation*}
    \lim_{k \to \infty}c_k - (1-p)c_k^2 = c - (1-p)c^2.
\end{equation*}
On the other hand the right side of \eqref{eq:Steveisloud}, call it $R_{k}$  is such that $\lim_{k\to \infty} R_{k}= p(2-c)c$. Indeed, for each $\epsilon > 0$, there exists a $K$ such that $|c - c_K| < \epsilon$.  Choose any $k \geq 2K$. We have:
\begin{equation*}
\begin{split}
    |R_k - p(2-c)c| &\leq p|c - c_{k-1}| + p|(c_1 - c_2)c_{k-1} + \cdots +(c_{k-2} - c_{k-1})c_2 - (1-c)c| \\ 
    &\leq p\epsilon + pc|c-c_K| + p|c_1 - c_K|\sup_{k-K+1 \leq j \leq k-1}|c_j - c| \\
    &\quad + p|c_{K+1} - c_{k-1}|\sup_{1 \leq j \leq k-K}c_j \\ 
    &\leq p\epsilon + pc\epsilon + p\epsilon + p\epsilon \leq 4\epsilon.
\end{split}
\end{equation*}
Since this holds for all $\epsilon$, we have that $R_k$ converges to $p(2-c)c$.  Consequently:
\begin{equation*}
    c - (1-p)c^2 = p(2-c)c \Leftrightarrow c \in \{0,1\}.
\end{equation*}
By Lemma \ref{fskdo}, we cannot have $c = 1$.  Thus, $c = 0$. This proves tightness of the sequence $X_{N}$.
\end{proof}

Lastly, we briefly analyze the case where $p > 1/2$. Let the children of $X_{N+1}$ be $X_N$ and $X'_N$. One has
\begin{equation*}
    \mathbb{E}[X_{N+1}] = a\mathbb{E}[X_N + X'_N] + (1-a)\mathbb{E}[\min(X_N,X'_N)] \geq 2a\mathbb{E}[X_N].
\end{equation*}
Since $2a > 1$, we have exponential growth for the expectation of $X_{N}$:
\begin{equation*}
    \mathbb{E}[X_N] \geq (2p)^N.
\end{equation*}

We end this section with a few questions of interest. 
\begin{enumerate}[(a)]
\item Consider the merging/annihilation process starting with $n$ particles of mass $1$. What is the collision order that maximizes the final total mass? If $n=2^{N}$, this should be given by a binary tree.  In other words, what is the tree that maximizes $\mathbb E X_{n}$?
\item What is the final mass when the motion of particles is random? For instance, what if particles perform a random walk on $\mathcal T$  starting at different locations? How about on $\mathbb Z$, or on $\mathbb Z^{2}$? 
\item Study the same problem on a $d$-ary tree. Is there a transition at $p=1/d$?
\item Instead of mass $1$ for each particle, start the process with independent masses with common distribution $F$. How does $X_{N}$ depend on $F$? Simulations suggest that the bottom of the support plays a major role determining the value of $X_{N}$.   
\end{enumerate}

\end{document}

%% file: Figure1.pdf_tex.tex
\begingroup%
  \makeatletter%
  \providecommand\color[2][]{%
    \errmessage{(Inkscape) Color is used for the text in Inkscape, but the package 'color.sty' is not loaded}%
    \renewcommand\color[2][]{}%
  }%
  \providecommand\transparent[1]{%
    \errmessage{(Inkscape) Transparency is used (non-zero) for the text in Inkscape, but the package 'transparent.sty' is not loaded}%
    \renewcommand\transparent[1]{}%
  }%
  \providecommand\rotatebox[2]{#2}%
  \ifx\svgwidth\undefined%
    \setlength{\unitlength}{215.80866699bp}%
    \ifx\svgscale\undefined%
      \relax%
    \else%
      \setlength{\unitlength}{\unitlength * \real{\svgscale}}%
    \fi%
  \else%
    \setlength{\unitlength}{\svgwidth}%
  \fi%
  \global\let\svgwidth\undefined%
  \global\let\svgscale\undefined%
  \makeatother%
  \begin{picture}(1,0.45781294)%
    \put(0,0){\includegraphics[width=\unitlength,page=1]{Figure1.pdf}}%
    \put(0.06,0.2){\color[rgb]{0,0,0}\makebox(0,0)[lb]{\smash{$1$}}}%
    \put(-0.02,-0.04){\color[rgb]{0,0,0}\makebox(0,0)[lb]{\smash{$1$}}}%
        \put(0.06,-0.04){\color[rgb]{0,0,0}\makebox(0,0)[lb]{\smash{$1$}}}%
                \put(0.115,-0.04){\color[rgb]{0,0,0}\makebox(0,0)[lb]{\smash{$1$}}}%
                        \put(0.195,-0.04){\color[rgb]{0,0,0}\makebox(0,0)[lb]{\smash{$1$}}}%
                                \put(0.25,-0.04){\color[rgb]{0,0,0}\makebox(0,0)[lb]{\smash{$1$}}}%
                                        \put(0.33,-0.04){\color[rgb]{0,0,0}\makebox(0,0)[lb]{\smash{$1$}}}%
                                                \put(0.383,-0.04){\color[rgb]{0,0,0}\makebox(0,0)[lb]
                     {\smash{$1$}}}%
                                                        \put(0.46,-0.04){\color[rgb]{0,0,0}\makebox(0,0)[lb]{\smash{$1$}}}%
                                                                \put(0.52,-0.04){\color[rgb]{0,0,0}\makebox(0,0)[lb]{\smash{$1$}}}%
                                                                   \put(0.59,-0.04){\color[rgb]{0,0,0}\makebox(0,0)[lb]{\smash{$1$}}}
                                                                      \put(0.65,-0.04){\color[rgb]{0,0,0}\makebox(0,0)[lb]{\smash{$1$}}}
                                                                         \put(0.725,-0.04){\color[rgb]{0,0,0}\makebox(0,0)[lb]{\smash{$1$}}}
                                                                            \put(0.785,-0.04){\color[rgb]{0,0,0}\makebox(0,0)[lb]{\smash{$1$}}}   
                                                                            \put(0.855,-0.04){\color[rgb]{0,0,0}\makebox(0,0)[lb]{\smash{$1$}}}  
                                                                             \put(0.92,-0.04){\color[rgb]{0,0,0}\makebox(0,0)[lb]{\smash{$1$}}}   
                                                               
          \put(0.99,-0.04){\color[rgb]{0,0,0}\makebox(0,0)[lb]{\smash{$1$}}}

                                       \put(0.38,0.2){\color[rgb]{0,0,0}\makebox(0,0)[lb]{\smash{$1$}}} 
                                       \put(0.59,0.2){\color[rgb]{0,0,0}\makebox(0,0)[lb]{\smash{$2$}}} 
                                       \put(0.92,0.2){\color[rgb]{0,0,0}\makebox(0,0)[lb]{\smash{$4$}}}  
                                       \put(0.52,0.48){\color[rgb]{0,0,0}\makebox(0,0)[lb]{\smash{$4$}}}                         
    \put(0.18,0.065){\color[rgb]{0,0,0}\makebox(0,0)[lb]{\smash{$1$}}}
    \put(-0.01,0.065){\color[rgb]{0,0,0}\makebox(0,0)[lb]{\smash{$1$}}}
    \put(0.26,0.065){\color[rgb]{0,0,0}\makebox(0,0)[lb]{\smash{$2$}}}
    \put(0.45,0.065){\color[rgb]{0,0,0}\makebox(0,0)[lb]{\smash{$1$}}}
    \put(0.527,0.065){\color[rgb]{0,0,0}\makebox(0,0)[lb]{\smash{$2$}}}
    \put(0.715,0.065){\color[rgb]{0,0,0}\makebox(0,0)[lb]{\smash{$2$}}}
    \put(0.79,0.065){\color[rgb]{0,0,0}\makebox(0,0)[lb]{\smash{$2$}}}
    \put(0.99,0.065){\color[rgb]{0,0,0}\makebox(0,0)[lb]{\smash{$2$}}}
    \put(0.97,0.065){\color[rgb]{0,0,0}\makebox(0,0)[lb]{\smash{$$}}}%
    \put(0.61285595,0.35){\color[rgb]{0,0,0}\makebox(0,0)[lb]{\smash{$$}}}%
    \put(0.78603823,0.34){\color[rgb]{0,0,0}\makebox(0,0)[lb]{\smash{$2$}}}%
    \put(0.18603823,0.34){\color[rgb]{0,0,0}\makebox(0,0)[lb]{\smash{$2$}}}
  \end{picture}%
\endgroup%